\documentclass[12pt]{amsart}

\usepackage{amsmath, amssymb, amsthm}

\usepackage{amssymb}

\newcommand{\K}{\mathcal{K}}
\newcommand{\M}{\mathfrak{M}}

\newcommand{\rank}{\mathrm{Rank}}
\newcommand{\chii}{\raisebox{2pt}{$\chi$}}
\newcommand{\C}{\mathbb{C}}
\newcommand{\tens}{\otimes}
\newcommand{\dsum}{\oplus}
\newcommand{\bigdsum}{\bigoplus}
\newcommand{\fixlist}{\mbox{} \@afterheading}
\newcommand{\proofnewline}{\mbox{} \\ \indent}

\newcommand{\iso}{\cong}
\newcommand{\dunion}{\amalg}

\newcommand{\labelledthing}[2]{\hspace{4pt}\buildrel {#2} \over #1 \hspace{3pt}} 

\newcommand{\labelledrightarrow}{\labelledthing{\longrightarrow}}
\usepackage{hyperref}

\usepackage{fullpage}
\newcounter{theorem}
\newtheorem{thm}[theorem]{Theorem}
\newtheorem{lemma}[theorem]{Lemma}
\newtheorem{prop}[theorem]{Proposition}
\newtheorem{cor}[theorem]{Corollary}
\newtheorem{defn}[theorem]{Definition}

\newcounter{proofpartcounter}

\theoremstyle{remark}
\newtheorem*{remark*}{Remark}
\newtheorem{remark}[theorem]{Remark}
\newtheorem{example}[theorem]{Example}

\numberwithin{equation}{section}
\numberwithin{theorem}{section}

\newcommand{\alabel}{\label}

\newcommand{\RP}{\mathrm{RP}_{p.w.}}
\newcommand{\PIPD}{\mathrm{PIPD}_{p.w.}}

\newcommand{\ltwo}{\ell_2}
\newcommand{\BH}{B(\ell_2(\mathbb N))}

\newcommand{\interior}[1]{{#1}^\circ}

\newcommand{\closure}{\overline}
\newcommand{\waysubset}{\subseteq\!\subseteq}
\newcommand{\hered}[1]{\mathrm{Her}\left(#1\right)}

\newcommand{\Cu}{\mathcal{C}u}

\newcommand{\CEIleq}{\preceq_{Cu}}
\newcommand{\CEIeq}{\sim_{Cu}}
\newcommand{\mvnleq}{\preceq}
\newcommand{\mvneq}{\cong}
\newcommand{\limcoh}[1]{\overleftarrow{\check{H}}^2(#1)}
\newcommand{\limchern}{\overleftarrow{c}_1}
\newcommand{\specialCuntz}{\widehat{\Cu}}
\newcommand{\specialV}{V_c}
\newcommand{\claim}{\noindent \textbf{Claim.}\ \ }
\renewcommand{\emptyset}{\varnothing}
\newcommand{\dimC}{\dim_\mathbb{C}}
\renewcommand{\M}{M}
\renewcommand{\setminus}{\backslash}

\newcommand{\proofparts}{\setcounter{proofpartcounter}{0}}
\newcommand{\ppartgeneral}[1]{\refstepcounter{proofpartcounter}\emph{#1 \arabic{proofpartcounter}.}\ \ }
\newcommand{\pstep}{\ppartgeneral{Step}}
\newcommand{\pcase}{\ppartgeneral{Case}}

\renewcommand{\rank}{\mathrm{rank}\ }

\title{Hilbert C$^*$-modules over a commutative C$^*$-algebra}
\author{Leonel Robert \and Aaron Tikuisis}
\thanks{The second author was supported by an NSERC CGS-D scholarship}


\begin{document}

\maketitle

\begin{abstract}
This paper studies the problems of embedding and isomorphism for countably generated Hilbert C$^*$-modules over commutative C$^*$-algebras.
When the fibre dimensions differ sufficiently, relative to the dimension of the spectrum, we show that there is an embedding between the modules.
This result continues to hold over recursive subhomogeneous C$^*$-algebras.
For certain modules, including all modules over $C_0(X)$ when $\dim X \leq 3$, isomorphism and embedding are determined by the restrictions to the sets where the fibre dimensions are constant.
These considerations yield results for the Cuntz semigroup, including a computation of the Cuntz semigroup for $C_0(X)$ when $\dim X \leq 3$, in terms of cohomological data about $X$.
\end{abstract}

\section{Introduction}
Hilbert C$^*$-modules are generalizations of Hilbert spaces where the coefficient space
is allowed to be a C$^*$-algebra. Hilbert C$^*$-modules appear naturally in many areas of C$^*$-algebra theory, such as
KK-theory, Morita equivalence of C$^*$-algebras, and completely 
positive operators. In \cite{CowardElliottIvanescu}, Coward, Elliott, and Ivanescu give a description of  the Cuntz semigroup of a C$^*$-algebra in terms of the Hilbert C$^*$-modules over the algebra. This ordered semigroup has been shown to be a key ingredient in the Elliott program for the classification of C$^*$-algebras
(see \cite{CiupercaElliott},\cite{ElliottToms},\cite{Toms:annals}). These applications of the Cuntz semigroup motivate the present work.

The focus of this paper is the class of countably generated Hilbert C$^*$-modules over a commutative C$^*$-algebra.
When the C$^*$-algebra is commutative, Hilbert C$^*$-modules may be alternatively described as fields of Hilbert spaces
over the spectrum of the algebra \cite{Alonso}.  
Here, we do not study fields of Hilbert spaces directly, since the Hilbert C$^*$-module setting relates more naturally to the applications that we have in mind to Hilbert C$^*$-modules over sub-homogeneous C$^*$-algebras and their inductive limits.
Some generality is lost by doing this, since the base space is then
restricted to be locally compact and Hausdorff.

The results here address the following  questions: when are two given
Hilbert C$^*$-modules isomorphic, and when does one embed in the other?
In the context of fields of Hilbert spaces, these questions were considered by Dixmier
and Douady in \cite{DixmierDouady}, and more systematically by Dupr\'e in \cite{Dupre:ClassI}, \cite{Dupre:infinite}, and \cite{Dupre:ClassII}. 
Our approach is based on a representation of a Hilbert C$^*$-module as a supremum of vector bundles
supported on a family of open sets that cover the space. Thus, our results parallel---and rely on---the
theory of locally trivial vector bundles (when the dimension of the fibres is constant, the field
of Hilbert spaces corresponding to a Hilbert module is in fact a locally trivial vector bundle).

A fundamental result in the theory of vector bundles over a space $X$ of finite dimension states that when the fibre dimension of one bundle is sufficiently smaller than that of another one, the one bundle embeds into the other.
In Theorem \ref{DimFinite-Embedding}, we generalize this to countably generated Hilbert $C_0(X)$-modules:
if $M,N$ are countably generated Hilbert $C_0(X)$-modules, and
\[ \dim N|_x \geq \dim M|_x + \frac{\dim X - 1}{2} \quad \text{for all $x \in X$}, \]
then $M$ embeds into $N$.
Here, $M|_x$ denotes the fibre of $M$ at $x\in X$.
This result strengthens \cite{Toms:comparison}, where it is shown that $M$ is Cuntz below $N$ (see also \cite[Proposition 7]{Dupre:ClassII} for the case that $M$ has constant dimension and $N$ has infinite dimensional fibres on an open set).
In Corollary \ref{DimFinite-Embedding-RSH} we show that this result continues to hold for Hilbert C$^*$-modules over recursive subhomogeneous  C$^*$-algebras.

In order that two given Hilbert $C_0(X)$-modules $M$ and $N$ be isomorphic, the dimensions of $M|_x$ and $N|_x$ must agree for all $x\in X$.
Furthermore, there must be an isomorphism between the vector bundles arising by restricting $M$ and $N$ to the sets of constant dimension. In Theorem \ref{Dim3-Embedding-Iso}, we find certain situations in which this is the only obstruction to
the modules being isomorphic. This is the case, for example, when $X$ has dimension at most 3.
In fact, for spaces of dimension at most 3, the ordered semigroup of isomorphism classes of Hilbert C$^*$-modules may be
described in terms of cohomological data extracted from the module (this result is obtained
in \cite[Corollary 1]{Dupre:ClassII} for modules of finite order when $X$ has dimension at most 2).
In Example \ref{clutching-sphere}, we show that this classification does not extend to spaces of dimension larger than 3.
However, for spaces of larger dimension,
we show that,  if we have an isomorphism between the vector bundles arising by
restricting $M$ and $N$ to the sets of constant dimension, then
\[ M^{\dsum \lceil\frac{\dim X}{2}\rceil} \iso N^{\dsum \lceil\frac{\dim X}{2}\rceil} . \]

In Section \ref{CEI} we consider the Cuntz comparison of Hilbert C$^*$-modules and the Cuntz
semigroup of $C_0(X)$. Our results on embedding and isomorphism of Hilbert C$^*$-module
readily yield corollaries about the Cuntz comparison and equivalence of Hilbert C$^*$-modules.
We give a description of the Cuntz semigroup of $C_0(X)$, for $\dim X\leq 3$.
In Example \ref{telescope}, we resolve an outstanding question from \cite{CowardElliottIvanescu} of whether Cuntz comparison is the same as embedding: two Hilbert $C_0(X)$-modules are presented, with $X$ of dimension 2, which are Cuntz equivalent, yet such that neither one embeds in the other.
A peculiarity in the topological properties of $X$ permits this example.

In the last section we completely determine the group $K_0^*(C_0(X))$ originally considered by Cuntz: the dimension function suffices to 
determine the class of a Hilbert C$^*$-module in $K_0^*(C_0(X))$.
This demonstrates that the Cuntz semigroup is significantly more interesting than the group $K_0^*$.
We also prove that if the restriction of a countably generated Hilbert C$^*$-module $M$ to an open set $U$
is the trivial module with infinite dimensional fibres, i.e., $\ltwo(U)$, then 
$M\cong M\oplus \ltwo(U)$. As a corollary, we show that if $\dim X$ is finite, and the set 
of points where $M$ has infinite dimensional fibres is open, then the isomorphism class of $M$
is determined by its restriction to the set where it has finite dimensional fibres.
Conjecturally, this holds even if the set where the fibres are infinite dimensional is not open (under the hypothesis that $\dim X$ is finite).
This conjecture was put forth by Dupr\'{e} in \cite[Conjecture 1]{Dupre:infinite}, and our result is a partial confirmation.

The organization of the paper is as follows.
Section \ref{DimFinite} contains preliminaries, particularly the description of Hilbert $C_0(X)$-modules in terms of rank-ordered families of projections.
In Section \ref{Sec-DimFinite-Embedding}, we prove the result that when the pointwise dimensions of two Hilbert modules differ sufficiently then one embeds into the other.
The main result of Section \ref{Dim3} is that for $\dim X \leq 3$, isomorphism and embedding of Hilbert $C_0(X)$-modules depends only on the restrictions to their sets of constant dimension.
In Section \ref{CEI}, the results of Section \ref{Dim3} are applied to the Cuntz semigroup, ultimately producing a description of $\Cu(C_0(X))$ for $\dim X \leq 3$.
Section \ref{FurtherRemarks} contains distinct parts.
In \ref{Clutching}, an example is given of non-isomorphic Hilbert $C(S^4)$-modules whose restrictions to their sets of constant rank are isomorphic.
A computation of $K_0^*(C_0(X))$ for $\dim X < \infty$ is given in \ref{K0*}.
Finally, \ref{Absorption} contains the result that if the restriction of a Hilbert $C_0(X)$-module $M$ to an open set $U$ is isomorphic to $\ltwo(U)$ then $M$ is isomorphic to $M \dsum \ltwo(U)$.

\section{Preliminary definitions and results}\alabel{DimFinite}

\subsection{\texorpdfstring{Hilbert $C_0(X)$-modules}{Hilbert Co(X)-modules}}
A right Hilbert C$^*$-module  over a C$^*$-algebra $A$
is a right $A$-module $M$, endowed with
an $A$-valued inner product $\langle\cdot,\cdot\rangle$, and such that $M$ is complete with respect to the norm
$m\mapsto \|\langle m,m\rangle\|^{1/2}$. The reader is referred to \cite{Lance} for a more detailed 
definition of Hilbert C$^*$-module and for the general theory
of these objects. Here we review a few facts about Hilbert C$^*$-modules that will be 
used throughout the paper. We will often refer to Hilbert C$^*$-modules simply as Hilbert modules.
In the discussion that follows we assume that the C$^*$-algebra acts on the right of the Hilbert modules 
(this provision will be not necessary once we specialize to $A=C_0(X)$).

A Hilbert module is said to be  countably generated if it contains a countable set $\{m_i\}_{i=0}^\infty$
such that the sums  $\sum m_i\lambda_i$, with $\lambda_i\in A$, form a dense subset
of the module. For a Hilbert module $M$ we denote by $K(M)$ the C$^*$-algebra of compact operators
on $M$. By $\ltwo(A)$ we denote the Hilbert module over $A$ of sequences $(x_i)_{i=0}^\infty$, $x_i\in A$,
such that $\sum x_i^*x_i$ is norm convergent in $A$. It is known that every countably generated  Hilbert
module is isomorphic to one of the form $\overline{a\ltwo(A)}$, with $a\in K(\ltwo(A))^+$. 

For $a\in K(\ltwo(A))^+$ let us denote by $M_a$ the Hilbert module $\overline{a\ltwo(A)}$.
Let $\hered{a}$ denote the hereditary algebra generated by $a$ in $K(\ltwo(A))$, i.e., the algebra
$\overline{aK(\ltwo(A))a}$. 

Let $a,b\in K(\ltwo(A))^+$. If $a=s^*s$ and $\hered{ss^*}=\hered{b}$ for some $s\in K(\ltwo(A))$, 
then the map $\phi_s\colon M_a\to M_b$ given by
\[ \phi_s(|s|m):=sm \]
for $m\in \ltwo(A)$ (and extended continuously to all of $M_a = M_{|s|}$), is an isomorphism of Hilbert modules. Furthermore, if $\phi\colon M_a\to M_b$ is an isomorphism then $\phi=\phi_s$ for some
$s$ as above.

Let us now focus on the case of commutative C$^*$-algebras.
Henceforth, unless otherwise stated, $X$ will denote a locally compact Hausdorff space.
We will often speak of the dimension of $X$, by which we mean the covering dimension.
When specializing to the algebra $C_0(X)$, we have $\ltwo(C_0(X))\cong C_0(X,\ltwo(\mathbb{N}))$ and $K(\ltwo(C_0(X)))\cong C_0(X,K(\ltwo(\mathbb{N})))$, where $C_0(X,K(\ltwo(\mathbb{N})))$ acts pointwise on $C_0(X,\ltwo(\mathbb{N}))$.
In the sequel we will make the identifications given by these isomorphisms. We will denote the C$^*$-algebra $K(\ltwo(\mathbb{N}))$ simply as $\mathcal K$.

For $a\in C_0(X,\mathcal K)^+$, let $p\colon X\to \BH$ be the projection-valued map defined by
$p(x):=\chii_{(0,\infty)}(a(x))$ for all $x\in X$, i.e., $p(x)$ is the range projection of $a(x)$. We refer to $p$
as the pointwise range projection of $a$. For $s\in C_0(X,\mathcal K)$ let $v\colon X\to \BH$ be such that, 
for each $x\in X$, $s(x)=v(x)\big|s(x)\big|$ is the polar decomposition of $s(x)$. We refer to $v$
as the partial isometry arising from the pointwise polar decomposition of $s$. Suppose
that $a=s^*s$ and set $ss^*=b$. The module $M_a$ and the map $\phi_s\colon M_a\to M_b$ defined above can be neatly expressed in terms of $p$ and $v$ as follows.

\begin{lemma}
Let $a$, $p$, $s$, and $v$ be as in the previous paragraph. Then
\begin{align}
M_a =\{m\in \ltwo(C_0(X)) \mid p(x)m(x)=m(x) \hbox{ for all } x\in X\},& \quad \text{and} \alabel{Map} \\
(\phi_sm)(x) =v(x)m(x), \quad \hbox{for all $m\in M_a$ and $x\in X$}.&\alabel{phisv}
\end{align}
\end{lemma}

\begin{proof}
We clearly have the inclusion of $M_a$ in the right side of \eqref{Map}.
Let $m\in\ltwo(C_0(X))$ be such that $p(x)m(x)=m(x)$ for all $x\in X$. 
Since  $a^{1/n}(x)\to p(x)$ strongly for every $x$, 
we have  $\left\langle a^{1/n}(x)m(x),m(x)\right\rangle \nearrow \langle m(x),m(x)\rangle$ for every $x\in X$.
By Dini's Theorem this convergence is uniform on compact subsets of $X$. We thus have that $(1-a^{1/n})^{1/2}m\to 0$
in $\ltwo(C(X))$, and so $(1-a^{1/n})m\to 0$ in $\ltwo(C(X))$. Thus, $m\in M_a$. 

For \eqref{phisv}, we have $(\phi_s|s|m)(x)=s(x)m(x)=(v|s|)(x)m(x)$ for all $x\in X$. The vectors $|s|m$, with $m\in \ltwo(C_0(X))$, form a dense subset of $M_a$. Hence, $(\phi_sm)(x)=v(x)m(x)$ for all $m\in M_a$ and $x\in X$. 
\end{proof}

Since $M_a$ and $\phi_s$ depend only on $p$ and $v$ we will denote them by $M_p$
and $\phi_v$ when the relation between $a$ and $p$, and the relation
between $s$ and $v$, are understood. 

Let us denote by $\RP(X)$ the set of pointwise range projections of
elements in $C_0(X,\mathcal K)^+$. Let us denote by $\PIPD(X)$ the set of partial isometries
arising from the polar decomposition of an element in $C_0(X,\mathcal K)$.
It follows from the lemma and the remarks above that if $p,q\in \RP(X)$ then 
\begin{align*}
M_p\subseteq M_q &\Leftrightarrow p\leq q,\\
M_p\cong M_q     &\Leftrightarrow  \hbox{$p=v^*v$, $vv^*=q$ for some $v\in \PIPD(X)$.}
\end{align*}
In the latter case we write $p\cong q$.
If $M_p$ embeds into $M_q$, we write $p \preceq q$.

One can intuitively imagine what is meant by the restriction of a Hilbert $C_0(X)$-module to a subset of $X$, but let us give a formal definition.
If $F \subseteq X$ is a closed subset then let $M|_F := M/MC_0(X \setminus F)$, which is a $C_0(F)$-module ($C_0(F) \iso C_0(X)/C_0(X \setminus F)$ via the restriction map).
If $U \subseteq X$ is an open subset, we may let $M|_U := MC_0(U)$.
Combining these, if $Y \subseteq X$ is the intersection of a closed subset $F$ and an open subset $U$, we can see that $(M|_F)|_{U \cap F} = (M|_U)|_{U \cap F}$, and we define $M|_Y$ to be this.
For $x \in X$, we will, by abusing notation, allow $M|_x := M|_{\{x\}}$; since $C(\{x\}) \iso \mathbb{C}$, this is simply a Hilbert space, so that $\dim M|_x$ makes sense. 
Note that if $M$ is countably generated then so are $M|_F, M|_U$ for $F$ closed and for $U$ open and $\sigma$-compact.
One can easily check that for $p \in \RP(X)$, and for $Y$ the intersection of a closed and an open subset of $X$, we have
\[ M_p|_Y = M_{p|_Y}. \]

We denote by $\dim M$ the map from $X$ to $\mathbb N\cup \{\infty\}$
given by 
\[ \dim M(x):=\dim M|_x. \]
If $M=M_p$, and $a\in C_0(X)\otimes \K$ has pointwise
range projection $p$, then $\dim M(x)=\rank p(x)=\rank a(x)$ for all $x\in X$.

For a function $a\colon X \to \BH$, define the sets $R_{=i}(a), R_{\geq i}(a), R_{\leq i}(a)$ in terms of the function $\rank a$ as follows:
\begin{align*}
R_{=i}(a) &:= \{x \in X\mid \rank a(x) = i\}, \\
R_{\geq i}(a) &:= \{x \in X\mid \rank a(x) \geq i\}, \\
R_{\leq i}(a) &:= \{x \in X\mid \rank a(x) \leq i\}.
\end{align*}
Likewise, for a Hilbert $C_0(X)$-module $M$, we define $R_{=i}(M),R_{\geq i}(M), R_{\leq i}(M)$ in terms of $\dim M$ (so that, for example, $R_{=i}(M_p) = R_{=i}(p)$ for $p \in \RP(X)$).

\subsection{Rank-ordered families of projections}
Presently, we introduce forms of data that describe pointwise range projections of positive elements and partial isometries from pointwise polar decompositions. These data thus serve to describe countably generated
Hilbert modules and the embedding maps between them.

Throughout this paper, the phrase ``a continuous projection on the space $X$'' will mean a projection in the C$^*$-algebra $C_b(X,\K)$. Similarly, the phrase ``a continuous partial isometry on $X$'' refers to a partial isometry in $C_b(X,\K)$.

\begin{defn} \alabel{ROF-Defn} (cf.\ \cite{LinPhillips:mindiffeos})
A rank-ordered family of projections is a family of pairs $(p_i,A_i)_{i=0}^\infty$ such that
\begin{enumerate}
\item $X = \bigcup_{i} A_i$.
\item For each $i$, $\bigcup_{j \geq i} A_j = \bigcup_{j \geq i} \interior{A_j}$.
\item For each $i \geq 1$, $\bigcup_{j \geq i} A_j$ is a $\sigma$-compact subset of $X$. \alabel{ROF-SigmaCompact}
\item For each $i$, $p_i$ is a continuous projection on $A_i$ with constant rank $i$.
\item For $i \leq j$, $p_i \leq p_j$ on $A_i \cap A_j$. \alabel{ROF-Compatibility}
\end{enumerate}
A rank-ordered family of partial isometries is a family of pairs $(v_i, A_i)_{i=0}^\infty$ such that
the sets $(A_i)_{i=0}^\infty$ satisfy the conditions (1)-(3) above, and also 
\begin{enumerate}
\setcounter{enumi}{3}
\item For each $i$, $v_i$ is a continuous partial isometry with constant rank $i$.
\item For $i \leq j$, $v_j^*v_i=v_i^*v_i$ on $A_i \cap A_j$. \alabel{ROF-PIso-Compatibility}
\end{enumerate}
\end{defn}

In \cite{LinPhillips:mindiffeos}, the name ``rank-ordered family of projections'' is introduced to describe a similar object to what appears above; however, we caution that the objects are not exactly the same.
What appears in \cite{LinPhillips:mindiffeos} is a finite family of pairs $(p_i, A_i)_{i=0}^n$ for which the sets $A_i$ are required to be open. By using an infinite family, we allow the rank-ordered family of projections to describe range projections with possibly unbounded or even infinite rank. Also, we will often make use of rank-ordered families where the sets $A_i$ are not open, but rather, are relatively closed in $\bigcup_{j \geq i} A_j$.

\begin{prop}\alabel{ROF-Characterization} (cf. \cite[Lemma 3.1]{Phillips:arsh})
Let $X$ be a locally compact Hausdorff space.

(i)
Given a rank-ordered family of projections $(p_i,A_i)_{i=0}^\infty$, a projection in $\RP(X)$  (denoted $\bigvee p_i$) is defined by the formula
\[ \big( \bigvee p_i \big)(x) := \bigvee_{i\mid x \in A_i} p_i(x). \]
(If $x$ is only in finitely many sets $A_i$ then $\left(\bigvee p_i\right)(x) = p_i(x)$ for the greatest $i$ for which $x \in A_i$.)

Conversely, if $p\in \RP(X)$ then there exists a rank-ordered family of projections $(p_i,A_i)_{i=0}^\infty$ for which $p = \bigvee p_i$.

(ii)
Given a rank-ordered family of partial isometries $(v_i,A_i)_{i=0}^\infty$, a partial isometry in
$\PIPD(X)$ (denoted $\bigvee v_i$) is defined by 
\[
\big( \bigvee v_i \big)(x)=\lim_{\{i\mid x\in A_i\}} v_i(x),
\]
where the limit is taken in the strong operator topology.  (If $x$ is only in finitely many sets $A_i$ then $\left(\bigvee v_i\right)(x) = v_i(x)$ 
for the greatest $i$ for which $x \in A_i$.)

Conversely, if $v\in \PIPD(X)$, and $(p_i,A_i)_{i=0}^\infty$ is a rank-ordered family of projections such that $v^*v=\bigvee p_i$, then there exists a rank-ordered family of partial isometries $(v_i,A_i)_{i=0}^\infty$ for which $v = \bigvee v_i$
and $p_i=v_i^*v_i$ for all $i$.
\end{prop}

\begin{proof}\proofnewline
(i) Let $(p_i,A_i)_{i=0}^\infty$ be a rank-ordered family of projections.
To see that $\bigvee p_i$ is a projection in $\RP(X)$, we shall construct a positive element by the formula
\[ \sum \lambda_i p_i, \]
where $\lambda_i\colon X \to [0,2^{-i}]$ is a continuous function which is zero outside of $A_i$ and non-zero on $A_i \setminus \bigcup_{j > i} A_j$.
Such a sum converges in $C_0(X, \K)$ to an element whose pointwise range projection is exactly $\bigvee p_i$.

To see that the function $\lambda_i$ exists, we need to exhibit a $\sigma$-compact open set $U$ satisfying
\[ A_i \setminus \bigcup_{j > i} A_j \subseteq U \subseteq A_i, \]
for then $\lambda_i$ is given by a strictly positive element of $C_0(U)$.
The set $A_i \setminus \bigcup_{j > i} A_j$ is relatively closed within the $\sigma$-compact set $\bigcup_{j \geq i} A_j$.
By Urysohn's lemma, there exists a function $f\colon\bigcup_{j \geq i} A_j \to \mathbb{R}$ such that $f(A_i \setminus \bigcup_{j > i} A_j) = 1$ and $f(\bigcup_{j \geq i} A_j \setminus A_i) = 0$.
We may set $U = f^{-1}((\frac{1}{2}, \infty))$.
The set $U$ is open since it is relatively open within the open set $\bigcup_{j \geq i} A_i$.
It is $\sigma$-compact since it is a relatively $G_\delta$ subset of the $\sigma$-compact set $\bigcup_{j \geq i} A_j$.

Conversely, suppose $a \in C_0(X, \K)^+$ and $p(x) = \chii_{(0,\infty)}(a(x))$ for all $x$.
At each point $x \in X$, let $\sigma_1(a)(x), \sigma_2(a)(x), \dots$ be the list of eigenvalues of $a(x)$ in non-increasing order (so that $\sigma_i(a) \in C_0(X)$).
Then let $A_0 = X$ and 
\[ A_i = \{x \in X \mid  \sigma_i(a)(x) > \sigma_{i+1}(a)(x)\} \]
for $i \geq 1$.
Since each function $\sigma_i$ is continuous and vanishing at $\infty$, the sets $A_i$ are open and $\sigma$-compact.
By the choice of $A_i$, we may define $p_i(x)$ to be the spectral projection of $a(x)$ onto the $i$ greatest eigenvalues. It is clear that this definition makes $p_i$ continuous.
Moreover, since for each $x \in X$, $\sigma_i(x) \searrow 0$, we see that 
\[ R_{\geq i}(a) = \{x \mid \sigma_i(x) > 0\} = \bigcup_{j \geq i} A_j, \]
and from this it is easy to see that $p = \bigvee p_i$.

(ii) We can show that $\bigvee v_i$ is a partial isometry from a pointwise polar decomposition by the same argument used
to show that  $\bigvee p_i$ is in $\RP(X)$.

Conversely, for $v$ a partial isometry arising in the polar decomposition $s(x) = v(x)\big|s(x)\big|$ of some
$s\in C_0(X, \K)$, and $(p_i,A_i)_{i=0}^\infty$ 
a rank-ordered family for the pointwise range projection of $s^*s$, we can define $v_i(x) = v(x)p_i(x)$, for $x\in A_i$.
The resulting family $(v_i,A_i)_{i=0}^\infty$ is a rank-ordered family of partial isometries and it is easily verified
that $v(x)=\lim_{\{i|x\in A_i\}} v_i(x)$ for all $x\in X$.
\end{proof}

If $p\in \RP(X)$ and  $(p_i,A_i)_{i=0}^\infty$ is a rank-ordered family of  projections
such that $p=\bigvee p_i$, we say that $(p_i,A_i)_{i=0}^\infty$ is a rank-ordered family
for $p$. Similarly, if  $v\in \PIPD(X)$ and $(v_i,A_i)_{i=0}^\infty$ is such that 
$v=\bigvee v_i$ we say that  $(v_i,A_i)_{i=0}^\infty$ is a rank-ordered family for $v$.
From the definition of rank-ordered families, if $(p_i,A_i)_{i=0}^\infty$ is a rank-ordered
family for $p$, then so is  $(p_i,\interior{A_i})_{i=0}^\infty$, and similarly for the
rank-ordered families of partial isometries.

\begin{remark}\alabel{ROF-ConstructionRmk}
The proof of Proposition \ref{ROF-Characterization} (i) is constructive: given a rank-ordered family $(p_i, A_i)_{i=0}^\infty$, it produces a positive element $a \in C_0(X, \K)^+$, and given a positive element $a \in C_0(X, \K)^+$, it produces a rank-ordered family $(p_i, A_i)_{i=0}^\infty$.
Moreover, a close look at the constructions involved reveals that if we begin with a rank-ordered family $(p_i, A_i)_{i=0}^\infty$, obtain a positive element $a$, and then obtain a new rank-ordered family $(q_i, B_i)_{i=0}^\infty$, then the new rank-ordered family will almost coincide with the given one: $B_i \subseteq A_i$ and
\[ q_i = p_i|_{B_i}. \]
In addition, in the situation that $A_i$ is $\sigma$-compact and open, we may arrange that $B_i = A_i$.
\end{remark}

\subsection{A few technical lemmas}

If $a$ is either a projection in $\RP(X)$ given by a rank-ordered family $(p_i,A_i)_{i=0}^\infty$, or a partial isometry in $\PIPD(X)$ given by a rank-ordered family $(v_i,A_i)_{i=0}^\infty$, then we have
\[ R_{\geq i}(a) = \bigcup_{j \geq i} A_i. \]
Notice that if the sets $A_i$ are replaced by smaller sets $A_i'\subseteq A_i$ such that 
\begin{equation}
R_{\geq i}(a) = \bigcup_{j \geq i} \interior{A_i'} \quad \text{for all $i$,} \alabel{ROF-SetCondition}
\end{equation}
then  the resulting rank-ordered family 
$(p_i|_{A_i'},A_i')_{i=0}^\infty$ (or $(v_i|_{A_i'},A_i')_{i=0}^\infty$) is still a rank-ordered family for $a$.
The condition \eqref{ROF-SetCondition} is equivalent to the following two conditions: $A_i'$ is a neighbourhood of $R_{=i}(a)$ for all $i$, and 
\[ R_{=\infty}(a) = \limsup_i \interior{A_i} := \bigcap_{i = 1}^\infty \bigcup_{j \geq i} \interior{A_j}. \]

The observation that we may get a new rank-ordered family by slightly shrinking the sets $A_i$, combined with the following lemma, allows us to find for the elements of
$\RP(X)$ and $\PIPD(X)$ rank-ordered families whose sets $A_i$ are relatively closed in $R_{\geq i}(a)$.

\begin{lemma}\alabel{ShrinkingLemma}
Let $(U_i)_{i=1}^\infty$ be an open cover of $X$, such that for each $i$, $R_{\geq i} := \bigcup_{j \geq i} U_i$ is $\sigma$-compact.
Then there exists a cover $(A_i)_{i=1}^\infty$ of $X$ such that for each $i$, we have:
\begin{enumerate}
\item \alabel{ShrinkingLemma-CoveringCondition} $\bigcup_{j \geq i} A_j = \bigcup_{j \geq i} \interior{A_j} = R_{\geq i}$, and
\item $A_i$ is relatively closed in $\bigcup_{j \geq i} A_j$.
\end{enumerate}
\end{lemma}

\begin{proof}
Using notation that reflects the intended use of this lemma, we shall set $R_{=i} = R_{\geq i} \setminus R_{\geq i+1}$.
We shall find open sets $B_i$ for which the relative closure $A_i := \overline{B_i} \cap R_{\geq i}$ is contained in $U_i$, and satisfying
\[ \bigcup_{j \geq i} B_j = R_{\geq i}. \]

\claim
$U_j$ can be covered by open sets $V_\alpha^{(j)}$ which satisfy
\[ R_{=j} \subseteq V_\alpha^{(j)} \subseteq \closure{V_\alpha^{(j)}} \cap R_{\geq j} \subseteq U_j. \]

\begin{proof}[Proof of claim.]
Let $x \in U_i$ be given.
Since $R_{\geq i+1}$ is open, $R_{=i}$ is relatively closed in $R_{\geq i}$.
Since $R_{\geq i}$ is $\sigma$-compact, we may use Urysohn's lemma to obtain a continuous function $f\colon R_{\geq i} \to \mathbb{R}$ such that $f(R_{=i}) = f(x) = 1$ and $f(R_{\geq i} \setminus U_i) = 0$.
Set $V = f^{-1}(\frac{1}{2}, \infty)$.
It is easy to verify the required inclusions.
\end{proof}

We will create our sets $B_i$ as the finite unions of sets $V_\alpha^{(i)}$ from the claim.
Let us see first that we can find such sets $B_i$ that satisfy the condition \eqref{ShrinkingLemma-CoveringCondition} for $i=1$.

Since $R_{\geq 1}$ is $\sigma$-compact, let $(K_t)_{t=1}^\infty$ be a countable cover consisting of compact subsets.
$K_1$ is compact, so it may be covered by finitely many sets of the form $V_\alpha^{(i)}$.
This gives us some $n_1 \geq 1$ and sets $B_i$ which are finite unions of $V_\alpha^{(i)}$'s, for $i=1, \dots, n_1$, such that $K_1 \subseteq B_1 \cup \cdots \cup B_{n_1}$.
By requiring $B_i$ to be a nonempty union, we have $R_{= i} \subseteq B_i$ for each $i$.

Now, $K_2 \setminus (B_1 \cup \cdots \cup B_{n_1})$ is compact and contained in $R_{\geq n_1 + 1}$.
Thus, it is covered by finitely many sets of the form $V_\alpha^{(i)}$, with $i \geq n_1 + 1$.
Again, this allows us to obtain $n_2 \geq n_1$ and sets $B_i$ for $i=n_1 + 1, \dots n_2$ as above.
Continuing on, we will eventually cover all the sets $K_j$, and thus all of $R_{\geq 1}$.

Let us now label the sets $(B_i^{(1)})_{i=1}^\infty$, where the superscript $(1)$ denotes the fact that their union covers $R_{\geq 1}$.
We may likewise find a sequence of sets $(B_i^{(k)})_{i=k}^\infty$ such that each $B_i^{(k)}$ is a finite union of sets $V_\alpha^{(i)}$, and the sets cover $R_{\geq k}$.
If we now let $B_i = \bigcup_{k=1}^i B_i^{(k)}$, then since the union is finite, the relative closure in $R_{\geq i}$ is still a subset of $U_i$, and now
\[ R_{\geq k} = \bigcup_{i \geq k} B_i \]
for each $k$.
\end{proof}

\begin{remark} \alabel{ROF-HalfOpen}
Let  $p\in \RP(X)$  and  let $(p_i,U_i)_{i=0}^\infty$ be a rank-ordered family for $p$.
Restricting  to the interiors of sets $U_i$ if necessary, we may assume that the sets
$U_i$ are open. Then by Lemma \ref{ShrinkingLemma}, we may find sets $A_i\subseteq U_i$, relatively closed
in $R_{\geq i}(p)$, and such that $(p_i|_{A_i},A_i)_{i=0}^\infty$ is a rank-ordered family for $p$.
Likewise, for every partial isometry $v\in \PIPD(X)$ there exists a rank-ordered family of partial isometries $(v_i,A_i)_{i=0}^\infty$ for $v$ such that  $A_i$ is relatively closed in $R_{\geq i}(v)$ for all $i$. 
\end{remark}

In certain situations, it is desirable to obtain rank-ordered families which are compatible with certain given data.
The following two lemmas provide instances where this is possible.

\begin{lemma} \alabel{ROF-ProjExtension}
Let $X$ be a $\sigma$-compact locally compact Hausdorff space and $Y \subseteq X$ a closed subset.
Let $p$ be a projection in $\RP(X)$. Let $(p_i, B_i)_{i=0}^\infty$ be a rank-ordered family of projections for $p|Y$.
Then there exists a rank-ordered family $(\tilde{p}_i, A_i)_{i=0}^\infty$ for $p$ such that $A_i \cap Y \subseteq B_i$ and $\tilde{p}_i|_{A_i \cap Y} = p_i|_{A_i \cap Y}$  for each $i$.
\end{lemma}

\begin{proof}
Let $b \in C_0(Y, \K)^+$ be obtained from the rank-ordered family $(q_i, B_i)_{i=0}^\infty$ as in the proof of  Proposition \ref{ROF-Characterization} (i), so that the pointwise range projection of $b$ is $p|_Y$.
Then $b$ is strictly positive in the hereditary subalgebra $\{c \in C_0(Y, \K) \mid c = p|_Y c p|_Y\}$ of 
$C_0(Y,\K)$, and this hereditary subalgebra is the image under the quotient map $C_0(X,\K) \to 
C_0(Y,\K)$ of the singly generated hereditary subalgebra $\{c \in C_0(X, \K)\mid  c = pcp\}$.
Thus, $b$ lifts to a strictly positive element $a$ of $\{c \in C_0(X, \K)\mid  c = pcp\}$, ie.\ $a|_Y = b$.

Let $(\tilde{p}_i, A_i)_{i=0}^\infty$ be the rank-ordered family given from $a$ by Proposition \ref{ROF-Characterization} (i). 
It is clear that the construction in Proposition \ref{ROF-Characterization} (i) is natural, so that $(\tilde{p}_i|_{A_i \cap Y}, A_i \cap Y)_{i=0}^\infty$ is the rank-ordered family that would be given from $a|_Y$.
So, by Remark \ref{ROF-ConstructionRmk}, we see that $A_i \cap Y \subseteq B_i$ and
\[ \tilde{p}_i|_{A_i \cap Y} = p_i|_{A_i \cap Y}. \qedhere
\]
\end{proof}

\begin{lemma}\alabel{ROF-ProjExtension2}
Let $X$ be a $\sigma$-compact locally compact Hausdorff space.
Let $(Y_i)_{i=1}^n$ be a family of closed subsets of $X$.  Suppose that for every $i$ we are
given a continuous projection $p_i\colon Y_i\to \mathcal K$ such that 
\[
p_i|_{Y_i\cap Y_j}\leq p_j|_{Y_i\cap Y_j}
\]
for all $i$ and $j$.  Let $q$ be a projection in $\RP(X)$  such
that $p_i\leq q|_{Y_i}$  for all $i$. Then there is a rank-ordered family $(q_i,A_i)_{i=1}^\infty$
for $q$, such that $p_i|_{A_i\cap U_j}\leq q_j|_{A_i\cap U_j}$  for all $i$ and $j$. 
\end{lemma}

\begin{proof}
Let $\lambda\in C_0(X)^+$ be strictly positive.
Let us show that there is $b\in C_0(X,\mathcal K)^+$ with pointwise range projection equal to $q$ and such that
for every $i=1,2,\dots,m$  we have $b=\lambda p_i+b_i$, with $b_i\in C_0(Y_i,K)^+$ such that
$b_i$ has pointwise range  projection $q-p_i$ and $\|b_i(x)\|\leq \lambda(x)\hbox{ for all }x\in C_i$.

Let $\tilde b\in C_0(X,\mathcal K)^+$ have pointwise range projection $q$ and norm at most 1. 
Let us define $b$ on $Y_n$ by $b=\lambda p_n+b_n$, where  
$b_n=(1-p_n)\lambda \tilde b(1-p_n)$. Notice that on the set $Y_n\cap Y_{n-1}$ the element $b$ admits the decomposition $\lambda p_{n-1}+b_{n-1}$, where 
$b_{n-1}=\lambda(p_n-p_{n-1})+b_n$  has pointwise range projection equal to 
$q-p_{n-1}$ and satisfies that $\|b_{n-1}(x)\|\leq \lambda(x)$ for all $x\in Y_n\cap Y_{n-1}$. We proceed to define
$b$ on the set $Y_n\cup Y_{n-1}$ in the following way: extend $b_{n-1}$ from $Y_n\cap Y_{n-1}$ to $Y_{n-1}$
in such a way that its pointwise range projection is $q-p_{n-1}$ and such that $\|b_{n-1}(x)\|\leq \lambda(x)$ for all
$x\in Y_{n-1}$; set $b=\lambda p_{n-1}+b_{n-1}$ on $Y_{n-1}$. Now notice that on the set 
$(Y_n\cup Y_{n-1})\cap Y_{n-2}$  the element $b$ admits the decomposition $b=\lambda p_{n-2} +b_{n-2}$,
where $b_{n-2}\in C_0\left((Y_n\cup Y_{n-1})\cap Y_{n-2},\K\right)^+$ has pointwise range projection $q-p_{n-2}$ and $\|b_{n-2}(x)\|\leq \lambda(x)$ for all $x\in (Y_n\cup Y_{n-1})\cap Y_{n-2}$. As before, we extend $b_{n-2}$ to 
$Y_{n-2}$ such that these properties are preserved and set $b=\lambda p_{n-2}+b_{n-2}$ on $Y_{n-2}$. 
This process is continued until $b$ is defined on the set $\bigcup_{i=1}^n Y_i$. We then extend $b$ to an element in $C_0(X,\mathcal K)$ with pointwise range projection $q$. 

It is not hard to check that the rank-ordered family arising from $b$ (by the construction
in the proof of Proposition \ref{ROF-Characterization} (i)) has the properties stated in this lemma.
\end{proof}

\begin{prop}\alabel{ComplementedMultiplierProjs} (cf. \cite[Lemma 2.2]{PereraToms:recasting})
Let $p$ be a continuous projection and $q$ a pointwise range projection, such that $p \leq q$.
Then $q - p$ is a pointwise range projection.
\end{prop}

\begin{proof}
Let $a \in C_0(X,\K)^+$ such that $q(x) = \chii_{(0,\infty)}(a(x))$ for all $x \in X$.
Denoting by $1$ the function $X \to \mathcal{B}(\ltwo(\mathbb{N}))$ which is constantly the unit, consider the element
$b=pap + (1-p)a(1-p) \in C(X,\K)^+$. One easily verifies that $qbq=b$, whence
$\chii_{(0,\infty)}(b) \leq q$. On the other hand,
\[ a \leq a + (2p-1)a(2p-1) = 2(pap + (1-p)a(1-p))=2b, \]
and so $q = \chii_{(0,\infty)}(a)\leq \chii_{(0,\infty)}(b)$.
Hence,
\begin{align*} q = \chii_{(0,\infty)}(b) &=\chii_{(0,\infty)}\big(pap + (1-p)a(1-p)\big)\\
&= \chii_{(0,\infty)}(pap) + \chii_{(0,\infty)}\big((1-p)a(1-p)\big),
\end{align*}
because the elements $pap$ and $(1-p)a(1-p)$ are orthogonal.
Since $p \leq q$, it is clear that $\chii_{(0,\infty)}(pap) = p$.
It follows that
\[ q - p = \chii_{(0,\infty)}((1-p)a(1-p)), \]
as required.
\end{proof}

\section{The embedding of a Hilbert module into one of sufficiently larger dimension\alabel{Sec-DimFinite-Embedding}}
In this section we extend to countably generated Hilbert $C_0(X)$-modules the well-known fact that a vector bundle always embeds into another one with dimension at least $(\dim X-1)/2$ larger than that of the first one.
A partial generalization can be found in \cite[Proposition 7]{Dupre:ClassII}.

Our proof rests on a repeated application of the
result for vector bundles. It is essential, however, that the result for vector bundles be stated in
a relativized form, as in \cite[Proposition 1]{Dupre:ClassII} or in \cite[Proposition 4.2]{Phillips:rsh}, in the sense that when given an embedding of the bundles restricted to a closed subset of $X$, it provides an extension of the embedding. The result that we obtain for Hilbert C$^*$-modules---Theorem \ref{DimFinite-Embedding}---is again relativized in the same sense. This allows us to extend the result even further, to Hilbert $A$-modules where $A$ is a recursive subhomogeneous algebra. In doing this, we follow the line of reasoning used by Phillips in \cite[Theorem 4.5]{Phillips:rsh}, and by Toms in
\cite[Theorem 4.6]{Toms:comparison}, where analogous results are obtained for projections in the case of Phillips, and for Cuntz comparison of Hilbert modules in Toms's case (Toms uses the language of positive elements rather than Hilbert modules).

We shall now restate the embedding result for vector bundles, in the language of projections.

\begin{lemma}\alabel{Phillips-embedding}
Let $X$ be a finite dimensional $\sigma$-compact Hausdorff space with and let $Y \subseteq X$ be a closed subset.
Let $p,q\colon X \to \K$ be continuous projections such that for all $x \in X$,
\begin{equation}
\rank q(x) \geq \rank p(x) + \frac{\dim X - 1}{2}. \alabel{Phillips-RankDifference}
\end{equation}
Let $v\colon Y \to \K$ be a continuous partial isometry such that $v^*v = p|_Y$ and $vv^* \leq q|_Y$.
Then there exists a continuous partial isometry $\tilde{v}\colon X \to \K$ such that $\tilde{v}|_Y = v$, $\tilde{v}^*\tilde{v} = p$, and $\tilde{v}\tilde{v}^* \leq q$.
\end{lemma}

\begin{proof}
In \cite[Proposition 4.2]{Phillips:rsh}, this lemma appears under the hypotheses that $X$ is compact and
the projections $p$ and $q$ belong to a matrix algebra over $C(X)$.
Let us explain how to reduce the current version of the lemma to \cite[Proposition 4.2]{Phillips:rsh}.

Suppose first that $X$ is compact. Using that any projection in $C(X, \K)$ is Murray-von Neumann equivalent to a projection in $C(X, \M_n)$ for some $n$, the lemma is easily reduced to the
case when $p,q\in C(X,\M_n)$ for some $n$. This is then \cite[Proposition 4.2]{Phillips:rsh}.

Let us consider now the case when $X$ is $\sigma$-compact. Let $(X_n)_{n=1}^\infty$ be an increasing sequence of compact subsets of $X$, such that $X = \bigcup X_n$. For simplicity, allow $X_1 \subseteq Y$.
We will define $\tilde{v}$ on successively larger domains $Y \cup X_1, Y \cup X_2, \dots$.
In this manner, $\tilde{v}(x)$ is eventually defined for each $x \in X$.

On $Y \cup X_1 = Y$, we must set $\tilde{v} = v$.
Having defined $\tilde{v}$ on $Y \cup X_i$, we apply the case of the lemma established previously---where
the total space $X$ was compact---to extend $\tilde{v}|_{(X_i \cup Y) \cap X_{i+1}}$ to a continuous partial isometry on $X_{i+1}$. We have thus defined $\tilde{v}$ on $X_{i+1} \cup Y$.
Since both $X_{i+1}$ and $Y$ are closed, and $\tilde{v}$ is continuous when restricted to either of them, we see that $\tilde{v}$ is continuous on their union, so that the induction is complete.
\end{proof}

\begin{thm}\alabel{DimFinite-Embedding}
Let $X$ be a finite dimensional locally compact Hausdorff space and let $Y \subseteq X$ be a closed subset.
Let $M,N$ be countably generated Hilbert $C_0(X)$-modules such that, for all $x \in X \setminus Y$,
\[ \dim N|_x \geq \dim M|_x + \frac{\dim X - 1}{2}, \]
where $\infty \geq \infty$ is allowed.
Let $\phi\colon M|_Y \to N|_Y$ be an embedding of Hilbert $C_0(Y)$-modules.
Then there exists $\tilde{\phi}\colon M \to N$, an embedding of Hilbert modules, such that $\tilde{\phi}|_Y = \phi$.
\end{thm}

\begin{proof}
We may assume that $M=M_p$ and $N=M_q$ for some projections $p,q\in \RP(X)$. Then $\phi=\phi_v$, for some
$v\in \PIPD(Y)$. In terms of $p$, $q$, and $v$, we must show that if $\rank p+(\dim X-1)/2\leq \rank q$, and $p|_Y=v^*v$, $vv^*\leq q$, then there is $\tilde v\in \PIPD(X)$ such that $\tilde{v}|_Y = v$, $p=\tilde{v}^*\tilde{v}$ and $\tilde{v}\tilde{v}^* \leq q$. We will first prove this under the additional assumption that the projections $p$ and $q$ have finite, constant rank; then we will drop these assumptions, first on $q$ and then on $p$.

\proofparts
\pcase \alabel{DimFinite-Embedding-pqConstRank}
Let us assume that $p$ and $q$ each have finite, constant rank.
The result is trivial unless the rank of $q$ is at least $1$ (of course, unless $\dim X \leq 1$, this is forced by the rank comparison condition).
In this case $X$ must be $\sigma$-compact, and so this case follows from Lemma \ref{Phillips-embedding}.

\pcase \alabel{DimFinite-Embedding-pConstRank}
Next, let us prove the result under the assumption that $p$ has finite, constant rank, but allowing $q$ to be an arbitrary pointwise range projection.
Since $vv^* \leq q|_Y$, Proposition \ref{ComplementedMultiplierProjs} says that $q|_Y - vv^*$ is a pointwise range projection and thus has a rank-ordered family $(r_i, B_i)_{i=0}^\infty$.
By adding these projections to $vv^*$, we obtain a rank-ordered family of projections $(q_i, B_i)_{i=0}^\infty$ for $q|_Y$, such that $q_i \geq vv^*$ for each $i$.
Moreover, by applying Lemma \ref{ROF-ProjExtension}, we may extend the $q_i$'s; the result is a rank-ordered family of projections $(q_i,A_i)_{i=0}^\infty$ for $q$, with the property that
\[ vv^*|_{A_i \cap Y} \leq q_i|_{A_i \cap Y} \]
for each $i$.
Moreover, by Lemma \ref{ShrinkingLemma}, we may assume (by possibly shrinking the sets $A_i$) that $A_i$ is relatively closed in $R_{\geq i}(q)$, for each $i$.

We will define $\tilde{v}$ on successively larger domains, beginning with $Y$ (where $\tilde{v}$ is equal to $v$), and adding on sets $A_i$ with $i$ in increasing order.
We will require that $\tilde{v}^*\tilde{v} = p$, and on the set $A_i$, $\tilde{v}\tilde{v}^* \leq q_i$.
We will only need to use $i \geq \rank p + \lceil \frac{\dim X - 1}{2} \rceil =: i_0$, since $X \setminus Y$ is covered by the sets $A_i$ with such $i$.

Having defined $\tilde{v}$ on $Y' = Y \cup A_{i_0} \cup \cdots \cup A_{i-1}$, let us extend the definition to include the set $A_i$.
Note that $Y' \cap A_i$ is relatively closed in $A_i$.
Let us check that $\tilde{v}\tilde{v}^*|_{Y' \cap A_i} \leq q_i$.
On $Y \cap A_i$, we have $\tilde{v} = v$, so that $\tilde{v}\tilde{v}^* \leq q_i$.
On $A_j \cap A_i$ ($j = i_0, \dots, i-1$), we have $\tilde{v}\tilde{v}^* \leq q_j \leq q_i$.
So, by applying the result we proved in Case \ref{DimFinite-Embedding-pqConstRank}, we may continuously extend $\tilde{v}$ to $A_i$, such that $\tilde{v}\tilde{v}^*|_{A_i} \leq p|_{A_i}$ and $\tilde{v}\tilde{v}^*|_{A_i} \leq q_i$.

By completing the induction, we obtain $\tilde{v}$ defined on all of $X$ and satisfying the conclusion.

\pcase \alabel{DimFinite-Embedding-NoRestriction}
Finally, let us also remove the assumption that the rank of $p$ is constant, and prove the lemma in full generality.
Let us take a rank-ordered family of projections $(p_i,A_i)_{i=0}^\infty$ for $p$, such that each set $A_i$ is relatively closed in $R_{\geq i}(p)$.
For simplicity, let us assume that $A_0 = X$.
By Proposition \ref{ROF-Characterization} (ii), we have a rank-ordered family $(v_i, A_i \cap Y)_{i=0}^\infty$ for $v$, such that $v_i^*v_i = p_i|_{A_i \cap Y}$ for each $i$.

To obtain $v$ satisfying the conclusion, we will obtain a rank-ordered family $(\tilde{v}_i, A_i)_{i=0}^\infty$ which satisfy the following:

\indent(i) $\tilde{v}_i|_{A_i \cap Y} = v_i$,

\indent(ii) $\tilde{v}_i^*\tilde{v}_i = p_i$, and

\indent(iii) $\tilde{v}_i\tilde{v}_i^* \leq q|_{A_i}$. \\
Also, implicit in the requirement that $(\tilde{v}_i, A_i)_{i=0}^n$ is a rank-ordered family is the condition that $\tilde{v}_j|_{A_i \cap A_j}$ is compatible with $\tilde{v}_i|_{A_i \cap A_j}$ for $i \leq j$ (as in Definition \ref{ROF-Defn} \ref{ROF-PIso-Compatibility}.
We will obtain the $\tilde{v}_i$'s in increasing order of $i$.

Set $\tilde{v}_0=0$.
Having defined $\tilde{v}_0, \dots, \tilde{v}_{i-1}$, we will proceed to define $\tilde{v}_i$.
Our method will once again be to define $\tilde{v}_i$ on successively larger domains, beginning with $A_i \cap Y$ (where $\tilde{v}_i$ is equal to $v_i$), then adding sets $A_i \cap A_j$ as $j$ decreases.
When we have added all such sets, we will have defined $\tilde{v}_i$ on all of $A_i$ (since $A_0 = X$).

Having defined $\tilde{v}_i$ on $Y' = A_i \cap (Y \cup A_{i-1} \cup \cdots \cup A_{j+1})$, let us now define it on $A_i \cap A_j$.
Notice that $Y'$ is relatively closed in $A_i$, so that $Y' \cap (A_i \cap A_j)$ is relatively closed in $A_i \cap A_j$.
We can easily see that $\tilde{v}_i$ is compatible with $\tilde{v}_j$ on $Y' \cap (A_i \cap A_j)$, so that $\tilde{v}_i - \tilde{v}_j$ is a partial isometry defined on $Y' \cap (A_i \cap A_j)$, taking the constant-rank projection $(p_i - p_j)|_{Y' \cap (A_i \cap A_j)}$ to a subprojection of $(q - \tilde{v}_j\tilde{v}_j^*)|_{Y' \cap (A_i \cap A_j)}$.
Since $Y'$ contains $A_i \cap Y$, we have for all $x \in (A_i \cap A_j) \setminus Y'$ that
\[
\rank (q - \tilde{v}_j\tilde{v}_j^*)(x) \geq \rank (p_i - p_j) + \frac{\dim X - 1}{2}.
\]
Hence, we may apply the result proven in Case \ref{DimFinite-Embedding-pConstRank} to obtain a continuous partial isometry $w$ defined on $A_i \cap A_j$ which takes $p_i - p_j$ to a subprojection of $q - \tilde{v}_j\tilde{v}_j^*$, and which agrees on $Y' \cap (A_i \cap A_j)$ with $\tilde{v}_i - \tilde{v}_j$.
Thus, if we let $\tilde{v}_i|_{A_i \cap A_j} = \tilde{v}_j + w$, then this definition agrees on $Y' \cap (A_i \cap A_j)$ with the previous one and satisfies the requirements stated above.

Therefore, $\tilde{v}_i$ may  be defined on all of $A_i$ in a manner compatible with $\tilde{v}_j$ for $j \leq i$ (i.e. so that Definition \ref{ROF-Defn} \eqref{ROF-PIso-Compatibility} holds).
Upon defining $\tilde{v}_i$ for all $i$, we obtain the partial isometry $\tilde{v}$ by Proposition \ref{ROF-Characterization} (i), as required.
\end{proof}

\begin{cor}\alabel{DimFinite-Embedding-PositiveElements}
Let $X$ be a finite dimensional locally compact Hausdorff space and let $Y \subseteq X$ be a closed subset.
Let $a,b \in C_0(X,\K)^+$ such that for all $x \in X \setminus Y$,
\[ \rank b(x) \geq \rank a(x) + \frac{\dim X - 1}{2}, \]
where $\infty \geq \infty$ is allowed.
Let $s \in C_0(Y,\K)$ such that $s^*s = a|_Y$ and $ss^* \in \hered{b|_Y}$.
Then there exists $\tilde{s} \in C_0(X,\K)$ such that $\tilde{s}|_Y = s$, $\tilde{s}^*\tilde{s} = a$ and $\tilde{s}\tilde{s}^* \in \hered{b}$.
\end{cor}

\begin{proof}
Let $p = \chii_{(0,\infty)}(a), q = \chii_{(0,\infty)}(b)$, and let $s = v|s|$ be the polar decomposition of $s$.
Applying Theorem \ref{DimFinite-Embedding} to these, we obtain $\tilde{v}$ such that $\tilde{v}|_Y = v$, $\tilde{v}^*\tilde{v} = p$, and $\tilde{v}\tilde{v}^* \leq q$.
Then the conclusion holds by setting $\tilde{s} = va^{\frac{1}{2}}$.
\end{proof}

The direct application of Theorem \ref{DimFinite-Embedding} to Hilbert modules over commutative C$^*$-algebras is using the situation that $Y$ is empty---giving an automatic embedding of a Hilbert modules if there is sufficient difference in their dimensions.
The full force of Theorem \ref{DimFinite-Embedding} is used, however, to show the generalization of this result to  C$^*$-algebras with a recursive subhomogeneous decomposition by spaces of  finite topological dimension.
By \cite[Theorem 2.16]{Phillips:rsh}, in the separable case, these are the C$^*$-algebras for which there is a finite upper bound on the dimensions of irreducible representations and for which, for each $n$, the space of irreducible representations of dimension $n$ is finite dimensional.

Let $M$ be a countably generated Hilbert module over a C$^*$-algebra $R$.
Let $\pi\colon R\to \M_n$ be a finite dimensional irreducible representation
of $R$. We may consider the push forward $\pi^*(M):=M\otimes_\pi \M_n$ of $M$ by $\pi$. This is a
countably generated Hilbert module over $\M_n$. As such, it is isomorphic to a module of the form
\[ \overline{a\ell^2(\M_n),} \]
for some $a \in (\K \tens \M_n)^+ \iso \K^+$.
As a vector space over $\C$, such a module has dimension $n\cdot (\rank a)$.

For $n\in \mathbb{N}$, let us denote by $\mathrm{Prim}_n(R)$ the space of primitive ideals
of $R$ corresponding to irreducible representations of dimension $n$.
By $\mathrm{Prim}(R)$, we denote the space of all primitive ideals of $R$.

\begin{cor}\alabel{DimFinite-Embedding-RSH}
Let $R$ be a separable unital C$^*$-algebra. Suppose that there is $n_0\in \mathbb{N}$
such that all irreducible representations of $R$ have dimension at most $n_0$ and that
$\mathrm{Prim}(A)$ is finite dimensional.
Let $M$ and $N$ be countably generated Hilbert $R$-modules.
Suppose that for every $n$ and every irreducible representation
$\pi$ of $R$ of dimension $n$, we have
\begin{equation}\alabel{DimFinite-Embedding-RSH-RankDifference}
\frac{\dimC \pi^*(N)}{n}\geq \frac{\dimC \pi^*(M)}{n}+\frac{\dim \mathrm{Prim}_n(R)-1}{2},
\end{equation}
where $\infty \geq \infty$ is allowed. Then $M$ embeds into $N$.
\end{cor}

\begin{proof}
By \cite[Theorem 2.16]{Phillips:rsh}, $R$ has a recursive subhomogeneous decomposition
by spaces of finite dimension. We defer the definition of recursive subhomogeneous decomposition
to \cite[Definition 1.1]{Phillips:rsh}, and we use the notation given in Definition 1.2 there.

Taking $M \iso M_a$ and $N \iso M_b$ for $a,b \in (\K \tens R)^+$, it must be shown that there exists $s \in \K \tens R$ such that
\[ a = s^*s \quad\text{and}\quad ss^* \in \hered{b}. \]
This can be done by induction on the length $l$ of the decomposition of $R$.
Note that the condition \eqref{DimFinite-Embedding-RSH-RankDifference} translates into
\[ \rank \sigma(b)(x) \geq \rank \sigma(a)(x) + \frac{\dim X_i - 1}{2}. \]

If $\ell = 0$ then $R = \M_n \tens C(X)$ for some $X$, and so the result follows directly from Corollary \ref{DimFinite-Embedding-PositiveElements} with $Y = \emptyset$.
For $\ell \geq 1$, $R$ is given by the pullback diagram
\[ \begin{array}{ccl}
R & \rightarrow & C_\ell = \M_{n(\ell)} \tens C(X_\ell) \\
\Big\downarrow && \Big\downarrow \scriptstyle{f \mapsto f|_{X_\ell^{(0)}}} \\
R^{(\ell - 1)} & \labelledrightarrow{\rho} & C_\ell^{(0)} = \M_{n(\ell)} \tens C(X_\ell^{(0)}),
\end{array} \]
for some unital clutching map $\rho$.
Set $a',b' \in \K \tens R^{(\ell - 1)}$ to be the images of $a,b$ in the stabilization of the $(\ell-1)$-stage algebra.
By induction, there exists $s' \in \K \tens R^{(\ell - 1)}$ such that
\[ a' = s^{\prime*}s' \quad\text{and}\quad s's^{\prime*} \in \hered{b'}. \]

Now, set $a^{\prime\prime}, b^{\prime\prime}$ to be the images of $a,b$ in $\K \tens C_\ell$.
We have
\begin{align*}
\rho(s')^*\rho(s') = \rho(a') = a^{\prime\prime}|_{X_\ell^{(0)}} \text{ and }
\rho(s')\rho(s')^* \in \hered{\rho(b')} = \hered{b^{\prime\prime}\big|_{X_\ell^{(0)}}}.
\end{align*}
By Corollary \ref{DimFinite-Embedding-PositiveElements}, we can extend $\rho(s')$ to an element $s^{\prime\prime} \in \K \tens C_\ell$, such that
\[ s^{\prime\prime*}s^{\prime\prime} = a^{\prime\prime} \quad\text{and}\quad s^{\prime\prime}s^{\prime\prime*} \in \hered{b^{\prime\prime}}. \]
Thus, we obtain $s := (s', s^{\prime\prime}) \in R$ satisfying $a = s^*s$ and $ss^* \in \hered{b}$, as required.
\end{proof}

A Hilbert $A$-module $M$ is said to be finitely generated if there is a finite subset of $M$ whose $A$-span is dense in $M$.

\begin{cor}\alabel{DimFinite-FinGenModules}
Let $R$ be a separable unital C$^*$-algebra. Suppose that there is $n_0\in \mathbb{N}$
such that all irreducible representations of $R$ have dimension at most $n_0$ and that
$\mathrm{Prim}(R)$ is finite dimensional.
For a countably generated Hilbert $R$-module $M$, the following are equivalent:

(i) $M$ is finitely generated, 

(ii) there is a uniform finite bound on $\dimC \pi^*(M)$ over all $\pi\in \mathrm{Prim}_n(R)$,
over all $n\leq n_0$. 

(iii) $M$ is isomorphic to a Hilbert submodule of $R^n$ for some $n\in \mathbb{N}$.
\end{cor}

\begin{proof}\proofnewline
(i)$\Rightarrow$(ii): If $M$ is finitely generated then $\pi^*(M)$ is finitely generated
(with the same number of generators) over $\M_{n}$. Thus, $\dim_{\C} \pi^*(M)$
is uniformly bounded over all $\pi$.

(ii)$\Rightarrow$ (iii): Let $k$ be an upper bound on $\dim_\C \pi^*(M)$ over all
$\pi\in \mathrm{Prim}_n (R)$, over all $n\leq n_0$.
Let $d$ be an upper bound on $n\cdot (\dim \mathrm{Prim}_n(R)-1)/2$ over all $n$. Then $M$ emebds into $R^{k+d}$ by Corollary \ref{DimFinite-Embedding-RSH}.

(iii) $\Rightarrow$ (i):
This is known to hold for the countably generated Hilbert $A$-modules over any C$^*$-algebra $A$; let us review the argument.
If $M$ is countably generated then $K(M)$ is $\sigma$-unital, so that it has a strictly positive element, $T$.
Since $M \subseteq R^n$, we have that $K(M) \subseteq \M_n \tens R$.
$M$ is generated by the columns of $T$.
\end{proof}
\section{Large gaps and \texorpdfstring{$\dim X\leq 3$}{dim X at most 3}\alabel{Dim3}}
The results of this section apply to two classes of Hilbert modules:
modules with large gaps in their dimension function and modules over a space
of dimension at most 3, where large gaps means gaps of at least $\dim X/2$, as defined below.

Let $\mathrm{Lsc}_\sigma(X,\mathbb{N}\cup \{\infty\})$ denote
the functions $f\colon X\to \mathbb{N}\cup \{\infty\}$ such that $f^{-1}((n,\infty])$
is open and $\sigma$-compact for all $n\geq 0$. Notice that if $M$ is a countably generated
Hilbert module then $M\cong M_a$ for some $a\in C_0(X,\K)^+$, and so 
$\dim M=\rank a\in \mathrm{Lsc}_\sigma (X,\mathbb{N}\cup \{\infty\})$.

\begin{defn}
Let $c\in \mathbb{R}$ and $f\in \mathrm{Lsc}(X,\mathbb{N}\cup\{\infty\})$.
Let us say that $f$ has gaps of at least $c$ if
$f(x)\neq f(y)\Rightarrow |f(x)-f(y)|\geq c$ for all $x,y\in X$.
\end{defn}

Here is a restatement  of Phillips's \cite[Proposition 4.2 (2)]{Phillips:rsh}. As for
our restatement of \cite[Proposition 4.2 (1)]{Phillips:rsh} (Lemma \ref{Phillips-embedding}), we have made some modifications to the original
statement: the space $X$ is assumed to be $\sigma$-compact instead of compact,
and the projections are taken in $C_b(X,\mathcal K)$ instead of in $\M_n(C(X))$.

\begin{lemma}(\cite[Proposition 4.2 (2)]{Phillips:rsh})\alabel{Phillips-Cancellation}
Let $X$ be a $\sigma$-compact, locally compact, Hausdorff space.
Let $Y$ a closed subset of $X$.
Let $p_1,p_2,q_1,q_2$ be  continuous projections on $X$ and $s$ and $w$ be continuous
partial isometries. Assume that $\rank p_i\geq \dim X/2$ and $p_i\perp q_i$ for $i=1,2$, and that
$q_1=s^*s$, $q_2=ss^*$, $p_1+q_1=w^*w$, and $p_2+q_2=ww^*$.

Further, let $v$ be a continuous partial isometry defined on $Y$ such that $p_1=v^*v$
and $p_2=vv^*$ on $Y$, and let $t\mapsto w_t$ be a continuous path of partial
isometries on $Y$ such that $w_t^* w_t = p_1 + q_1$ ,
$ w_tw_t^* = p_2 + q_2$, $w_0=w$ and $w_1=v+s$. Then there is
a continuous partial isometry $\tilde{v}$ on $X$ such that $\tilde{v}^*\tilde{v} = p_1$, $\tilde{v}\tilde{v}^*=p_2$  and  $\tilde{v}|_Y = v$,
and a continuous path $t\mapsto \tilde{w}_t$ of partial isometries on $X$ such that
$\tilde{w}_t^* \tilde{w}_t = p_1 + q_1$, $\tilde{w}_t \tilde{w}_t^* = p_2 + q_2$,   $\tilde{w}_0 = w$,    $\tilde{w}_1 = \tilde{v} + s$,  and
$\tilde{w}_t |_Y = w_t$.
\end{lemma}

For spaces of dimension 3, in order for $\rank p_i \geq \dim X /2$ to hold, rank one projections are excluded.
In the next lemma, we replace the homotopy condition in the previous lemma by a determinant-related condition, and by doing so, remove the restriction on the rank of the projection.

For a partial isometry $u\in C_b(X,\mathcal K)$
such that $u^*u=uu^*$ (i.e. $u$ is a unitary of the hereditary subalgebra generated by $u^*u$)
let us define $\det(u)\colon X\mapsto \mathbb{T}$ by $\det(u)(x):=\det (u(x)+1_{\mathcal{B}(\ltwo)}-(u^*u)(x))$. Notice that the determinant
on the right side is well defined since $u(x)-(u^*u)(x)$ has finite rank for all $x\in X$.

\begin{lemma}\alabel{Det1-Cancellation}
Let $X$ be a $\sigma$-compact, locally compact, Hausdorff space such that $\dim X\leq 3$.
Let  $Y$ a closed subset of $X$.
Let $p_1,p_2,q_1,q_2$ be  continuous projections on $X$ and $s$ and $w$ continuous
partial isometries. Assume that $\rank p_i\geq i$ and $p_i\perp q_i$ for $i=1,2$,
and that $q_1=s^*s$, $q_2=ss^*$, $p_1+q_1=w^*w$, and $p_2+q_2=ww^*$.

Further, let $v$ be a continuous partial isometry on $Y$ such that $p_1=v^*v$,
$p_2=vv^*$ and $\det (w^*(v+s))=1$ on $Y$. Then there is a continuous partial isometry $\tilde{v}$
on $X$ such that $p_1=\tilde{v}^*\tilde{v}$, $p_2=\tilde{v}\tilde{v}^*$, $\tilde{v}|_Y=v$, and $\det (w^*(\tilde{v}+s))=1$ on $X$.
\end{lemma}

\begin{proof}
We may assume without loss of generality that $p_1+q_1=p_2+q_2=w$.
Notice then that $v+s$ is a unitary of $\hered{(p_1+q_1)|_Y}$ and
that $\det (v+s)=1$ on $Y$.

Let $X_0$ denote the closed and open subset of $X$ where $\rank p_1=1$.
Let $r_1$ be a projection of constant rank 1 such that $r_1=p_1$ on $X_0$
and on $X\backslash X_0$, $r_1$ is any rank 1 subprojection of $p_1$ (the existence of
which is guaranteed by Lemma \ref{Phillips-embedding}). Let us write $p_1=p_1'+r_1$.

Since $\rank (vp_1')<\rank p_2$ on $Y$, we have by Lemma \ref{Phillips-embedding} that
$vp_1'$ extends to a partial isometry $v'$ on $X$ such that $ (v')^*v'=p_1'$ and $v'(v')^*\leq p_2$.
By the cancellation of projections on a space of dimension at most 3, there is a partial isometry $w_2$ on $X$ 
such that $r_1=w_2^*w_2$ and $w_2w_2^*=p_2-v'(v')^*$. Notice then that $v'+w_2+s$ is a unitary of $\hered{p_1+p_2}$
defined on $X$. By multiplying $w_2$ by a scalar function,  we may assume that
$\det(v'+w_2+s)(x)=1$ for all $x\in X$.

On the set $Y$ the unitary $(v+s)^*(v'+w_2+s)$
has the form $(p_1+p_2-r_1)+u'$, where $u'$ is a unitary in $\hered {r_1}$ such that
$\det(u')=1$. Since $\rank r_1=1$, we have that $u'=\det(u')r_1=r_1$.  Hence, $v+s=v'+w_2+s$ on $Y$, and
so $v'+w_2=v$ on $Y$. Setting  $\tilde{v}=v'+w_2$ we get the desired partial isometry on $X$.
\end{proof}

\begin{thm}\alabel{Dim3-Embedding-Iso} 
Let $M$ and $N$ be countably generated Hilbert modules over a finite dimensional space $X$.
Suppose that $\dim X\leq 3$ or that both $\dim M$ and $\dim N$ have gaps of at least $\dim X/2$.

(i) If $\dim M\leq \dim N$ and $M|_{R_{=i}(M)\cap R_{=j}(N)} \hookrightarrow N|_{R_{=i}(M)\cap R_{=j}(N)}$ for all $i, j$  then
$M\hookrightarrow N$.

(ii) If $\dim M= \dim N$ and $M|_{R_{=i}(M)}\cong N|_{R_{=i}(M)}$ for all $i$  then
$M\cong N$.
\end{thm}

\begin{proof}
Let $M=M_p$ and $N=M_q$ for some projections $p,q \in \RP(X)$.
In terms of these projections, we need to prove in part (i) that $p\preceq q$ and in part (ii) that $p\cong q$.

(i) Let us first consider the case that $\rank p$ has gaps of at least $\dim X/2$.
By Remark \ref{ROF-HalfOpen}, let $(p_i,A_i)_{i=0}^\infty$ be a rank-ordered family for $p$, where $A_i$ is relatively
closed in $R_{\geq i}(p)$ for all $i$.
Once again, we assume that $A_0 = X$.

\claim
We may shrink the sets $A_i$ such that $(p_i, A_i)_{i=0}^\infty$ is still a rank-ordered family for $p$ with $A_i$ relatively closed inside $R_{\geq i}$, and in addition, for $i < j < i + \dim X/2$,
\[ A_i \cap R_{= j} = \emptyset. \]

\begin{proof}[Proof of claim.]
We will use $B_i$ to denote the subset of $A_i$ that will replace $A_i$.
The indices $i$ may be divided into two groups: the ones for which $R_{=i} \neq \emptyset$, in which case $R_{=j} = \emptyset$ for $i < j < i + \dim X/2$, and the ones for which $R_{=i} = \emptyset$, in which case the purpose of $A_i$ is to satisfy
\[ \limsup A_i = R_{=\infty}. \]
For $i$ in the former group, we will allow $B_i = A_i$, while for $i$ in the latter group, we will need to arrange that $B_i$ is contained in $R_{\geq i + \dim X/2}$.
This will be achieved by using the same idea as the proof of Lemma \ref{ShrinkingLemma}.

For $i$ with $R_{=i} = \emptyset$, we can find a family $(C_\alpha^{(i)})$ of closed (in fact compact) subsets of $R_{\geq i + \dim X/2} \cap A_i$ whose interiors cover $R_{\geq i + \dim X/2} \cap \interior{A_i}$.
This family forms the eligible sets for $B_i$, which is to say that it suffices if we obtain $B_i$ as a finite union of $C_\alpha^{(i)}$'s.
For $i$ with $R_{=i} \neq \emptyset$, since we want $B_i = A_i$, we set the family $(C_\alpha^{(i)})$ of eligible sets to contain $A_i$ only.

We have chosen the families $(C_\alpha^{(i)})$ such that for each $i$,
\[ R_{\geq i} = \bigcup_{j \geq i} \bigcup_\alpha \interior{C_\alpha^{(j)}}. \]
As in the proof of Lemma \ref{ShrinkingLemma}, we may find sets $B_i$ which are finite unions of $C_\alpha^{(i)}$'s, such that $R_{\geq i} = \bigcup_{j \geq i} \interior{B_i}$, as required.
\end{proof}

We will construct by induction a rank-ordered family of partial isometries $(v_i,A_i)_{i=0}^\infty$ such that $p_i=v_i^*v_i$ and $v_iv_i^*\leq q$ for all $i$.

We set  $v_0=0$. Let us assume by induction that we have a rank-ordered family
$(v_i,A_i)_{i=0}^{k-1}$ with the desired properties. Let us set
$v_iv_i^*=p_i'$ for $i=0,1,\dots,k-1$. We will construct $v_k$ in two steps.

\proofparts
\pstep Let us take $m$ to be the least integer greater than or equal to $k$ for which $R_{=m} \neq \emptyset$.
We will first define $v_k$ on the set $R_{=k}(p)\cap R_{=m}(q)$.
On this set we have that $p\preceq q$
by hypothesis. Let $w_{k}$ be a continuous partial isometry on $R_{=k}(p)\cap R_{=m}(q)$
such that $p=w_{k}^*w_{k}$ and $q \geq w_{k}w_{k}^*$.
Let $\ell$ be the greatest integer less than $k$ for which $R_{=k}(p) \cap A_\ell \neq \emptyset$.
By the claim, we have $k - \ell \geq \dim X/2$, so that $p_{\ell}\cong p_{\ell}'$ and
\[
\rank (p-p_\ell) \geq \frac{\dim X}{2}
\]
on the set $R_{=k}(p)\cap R_{=m}(q)\cap A_{\ell}$.
Thus, by Lemma \ref {Phillips-Cancellation}---applied with $Y$ empty---there is a partial isometry $\tilde{v}_\ell$ taking $p-p_{\ell}$ to $w_kw_k^*-p_{\ell}'$ on  $R_{=k}(p)\cap R_{=m}(q)\cap A_{\ell}$, and such that $v_k := v_{\ell}+\tilde v_{\ell}$ is homotopic to $w_k$ on $R_{=k}(p)\cap R_{=m}(q)\cap A_{\ell}$.

We now proceed to extend $v_{k}$ to $R_{=k}(p)\cap R_{=m}(q)\cap A_{\ell - 1}$ in such a way that it is compatible
with $v_{\ell - 1}$ and is homotopic to $w_{k}$ on that set. Such extension is possible by
Lemma \ref{Phillips-Cancellation}. We continue in this way defining $v_{k}$ on the sets $R_{=k}(p)\cap R_{=m}(q)\cap (\bigcup_{j=i}^{\ell} A_{j})$ for $i=\ell,\ell-1,\dots, 0$. Since  $R_{=k}(p)\cap R_{=m}(q)\subseteq A_0=X$, this processs
results in $v_{k}$ being defined on $R_{=k}(p)\cap R_{=m}(q)$ and such that it is compatible with the partial isometries $(v_{i})_{i=0}^{k-1}$ and is homotopic to $w_{k}$.

\pstep We now look to extend $v_{k}$ to all of $A_{k}$.
Consider the set $A_k \cap A_{k-1}$.
The partial isometry $v_{k} - v_{{k-1}}$ is defined on
$A_{{k-1}}\cap A_{k} \cap R_{=k}(p) \cap R_{=m}(q)$ and implements the equivalence between $p_{k}-p_{{k-1}}$ and $q-p_{{k-1}}'$ on this set.
On the other hand, let us check that
\[ \rank (p_{k}-p_{{k-1}})+(\dim X-1)/2\leq \rank (q-p_{{k-1}}) \]
on $(A_{{k-1}}\cap A_{k})\backslash(R_{=k}(p)\cap R_{=m}(q))$.
Centainly, for $x \not\in R_{=k}(p)$ then $\rank q(x) \geq \rank p(x) \geq k + \\dim X /2$.
On the other hand, if $x \not\in R_{=m}(q)$ then $\rank q(x) \geq m + \dim X /2$.

Thus, by Theorem \ref{DimFinite-Embedding}, $v_{k}-v_{{k-1}}$ extends to a partial isometry $w$
on $A_{{k-1}}\cap A_{k}$ such that $p_{k} - p_{{k-1}}=w^*w$ and $ww^*\leq q-p_{k}$ on this set. We set $v_{k}=v_{{k-1}}+w$ on $A_{{k-1}}\cap A_{k}$. We continue extending $v_{k}$ to $A_{{k-2}}\cap A_{k}$, etc. Since $A_0\cap A_{k}=A_{k}$, by this process we obtain $v_{k}$ defined on $A_{k}$ and with the desired properties. This completes the induction step.

The proof of part (i) in the case that $\dim X\leq 3$ runs along similar lines as the above proof, but using Lemma \ref{Det1-Cancellation} in place of Lemma \ref{Phillips-Cancellation}.
Let us assume by induction that the rank-ordered family $(v_i,A_i)_{i=0}^{k-1}$ has already been defined and seek to define $v_k$. Only Step 1 of the above proof requires some modifications.

\emph{Step 1 (case $\dim X\leq 3$).} The partial isometry $w_k$ is chosen, as before,
implementing the equivalence between $p$ and $q$ on $R_{=k}(p)\cap R_{=k}(q)$.
We have $\rank  p-p_{k-1}\geq 1$ on $R_{=k}(p)\cap R_{=k}(q)\cap A_{k-1}$, which is sufficient for the application of Lemma \ref{Det1-Cancellation}. This ensures the existence of $\tilde v_{k-1}$
implementing the equivalence of $p-p_{k-1}$ with $q-p_{k-1}'$. We set $v_{k}=v_{k-1}+\tilde v_{k-1}$. Moreover, by multiplying $\tilde v_{k-1}$ by a scalar function, we may assume
that $w_k^*v_k$ is a unitary (in the hereditary algebra generated by
$p$ restricted to the set $R_{=k}(p)\cap R_{=k}(q)\cap A_{k-1}$) of determinant 1.
We continue extending $v_k$ to the sets  $R_{=k}(p)\cap R_{=k}(q)\cap A_{k-2}$, $R_{=k}(p)\cap R_{=k}(q)\cap A_{k-3}$,
etc, in such a way that it is compatible with the partial isometries $v_{k-2}$, $v_{k-3}$, etc,
and such that $\det (w_k^*v_k)=1$. That these extensions are possible is guaranteed by Lemma \ref{Det1-Cancellation}.

(ii) The proof of this part  applies equally well to the
two cases covered by the theorem: gaps of at least $\dim X/2$ in $\rank p$ and $\dim X\leq 3$.

Let $(p_i^{(1)},A_i^{(1)})_{i=0}^\infty$ be a rank-ordered family for $p$.
By part (i), we have a rank-ordered family of
partial isometries $(v_i,A_i^{(1)})_{i=0}^{\infty}$ such that $p_i^{(1)}=v_i^*v_i$ and $v_iv_i^*\leq q$.
Set $v_iv_i^*=\tilde p_i^{(1)}$.
Choose $n_1\in \mathbb{N}$ sufficiently large (how large will be specified later).
By Lemma \ref{ROF-ProjExtension2}, there is a rank-ordered family
$(q_i^{(1)},B_i^{(1)})_{i=n_1+1}^\infty$
for the restriction of $q$ to $R_{\geq n_1+1}(q)$,  such that $(\tilde p_i^{(1)},A_i^{(1)})_{i=0}^{n_1}\cup (q_i^{(1)},B_i^{(1)})_{i=n_1+1}^\infty$
is a rank-ordered family for $q$.
Choose $n_2>n_1$ large enough (how large will be specified later). By the proof of part (i), there are partial isometries $(w_i,B_i^{(1)})_{i=n_1+1}^{n_2}$ such that
$(v_i^*,A_i^{(1)})_{i=0}^{n_1}\cup (w_i,B_i^{(1)})_{i=n_1+1}^{n_2}$ is a compatible family of partial
isometries from $q$ to $p$. In this way, continue to build an intertwining between rank-ordered families of projections
for $p$ and $q$. On the side of $p$ the rank-ordered family of projections has the form
\[
(p_i^{(1)},A_i^{(1)})_{i=0}^{n_1}\cup (\tilde q_i^{(1)},B_i^{(1)})_{i=n_1+1}^{n_2}\cup (p_{i}^{(2)},A_i^{(2)})_{i=n_2+1}^{n_3}\dots
\]
and for $q$ we have
\[
(\tilde p_i^{(1)},A_i^{(1)})_{i=0}^{n_1}\cup (q_i^{(1)},B_i^{(1)})_{i=n_1+1}^{n_2}\cup
(\tilde p_i^{(2)},A_i^{(2)})_{i=n_2+1}^{n_3}\dots
\]
Let $r_i$, $i=1,2,\dots$, denote the increasing sequence of pointwise range projections below $p$ arising from the rank-ordered families 
\[(p_i^{(1)},\interior{A_i^{(1)}})_{i=0}^{n_1}, (p_i^{(1)},\interior{A_i^{(1)}})_{i=0}^{n_1}\cup (\tilde q_i^{(1)},\interior{B_i^{(1)}})_{i=n_1+1}^{n_2}, \hbox{ etc }.\]
Similarly define pointwise range projections $s_i$ below $q$. The rank-ordered families of
partial isometries constructed above give rise to an intertwining diagram of the form
\[
\begin{array}{ccccccccc}
M_{r_1} & \subseteq  & M_{r_2}      & \subseteq  & M_{r_3} & \subseteq &\cdots & \subseteq & M_p \\
\downarrow && \uparrow && \downarrow &&&& \\
M_{s_1}      & \subseteq  & M_{s_2} &\subseteq  & M_{s_3}       & \subseteq &\cdots & \subseteq & M_q,
\end{array}
\]
where the arrows indicate Hilbert C$^*$-module embeddings.
The indices $n_1,n_2$, $\dots$ are chosen such that the union of the submodules $M_{r_i}$, $i=1,2,\dots$, is dense in $M_p$,
and the union of the submodules $M_{s_i}$, $i=1,2,\dots$, is dense in $M_q$.
To do this, when choosing $n_i$, we can arrange that $R_{\geq k}(r_i)$ (or $R_{\geq k}(s_i)$) covers a given compact subset of $R_{\geq k}(p)$ for each $k \leq i$
By $\sigma$-compactness of the sets $R_{\geq k}$, the compact sets being covered may be chosen such that $\bigcup R_{\geq k}(r_i) = R_{\geq k}(p)$ for each $k$.
In this way, the intertwining
diagram above induces an isomorphism between the Hilbert modules $M_p$ and $M_q$.
\end{proof}

\begin{cor}
Let $M$ and $N$ be countably generated Hilbert $C_0(X)$-modules such that $\dim M=\dim N$ and
$M|_{R_{=i}(M)}\cong N|_{R_{=i}(M)}$ for all $i$. Then
\[
 M^{\dsum \lceil\frac{\dim X}{2}\rceil} \iso N^{\dsum \lceil\frac{\dim X}{2}\rceil}\iso
 M\oplus  N^{\dsum \lceil\frac{\dim X}{2}\rceil-1}. 
\]
\end{cor}
\begin{proof}
The modules  $M^{\dsum \lceil\frac{\dim X}{2}\rceil}$, $N^{\dsum \lceil\frac{\dim X}{2}\rceil}$, and $M\oplus  N^{\dsum \lceil\frac{\dim X}{2}\rceil-1}$ all have dimension functions with gaps of at least $\dim X/2$ and their restrictions to the sets of constant dimension are isomorphic. They are thus isomorphic by Theorem \ref{Dim3-Embedding-Iso} (ii).
\end{proof}

For the Hilbert modules in the previous result, it was shown that the isomorphism class depends only on the data given as the isomorphism classes of the restrictions to the sets of constant rank.
It is natural to ask what data of this form can be attained.
We answer this question in Proposition \ref{Dim3-DataAttained}.
The following technical tool will be needed in the proof of that proposition.

\begin{lemma}\alabel{DimFinite-Engulfing}
Let $X$ be a $\sigma$-compact locally compact Hausdorff space with finite covering dimension.
Suppose that we are given, for each $i = 0, \dots, n$, a continuous, rank $i$ projection
$p_i$ on a closed set $F_i$, such that the sets $F_i$ cover $X$ and the projections $p_i$ satisfy the compatibility condition $p_i \leq p_j$ on $F_i \cap F_j$ (this is the compatibility condition \eqref{ROF-Compatibility} required for a rank-ordered family of projections; the difference here is that the sets $F_i$ are closed).
Suppose we are also given a continuous, constant-rank projection $q$ on $X$, such that
\begin{equation} \rank q(x) \geq \rank \bigvee_{i \mid x \in F_i} p_i(x) + \frac{\dim X - 1}{2}. \alabel{DimFinite-Engulfing-RankDifference} \end{equation}
Let $Y \subseteq X$ be a closed subset, and let $v$ be a continuous partial isometry on $Y$ such that $v^*v = q|_Y, vv^* \geq \left(\bigvee_i p_i\right)|_Y$.
Then there exists a continuous partial isometry $\tilde{v}$ on $X$ such that $\tilde{v}|_Y = v$, $\tilde{v}^*\tilde{v} = q$ and $\tilde{v}\tilde{v}^* \geq \bigvee_i p_i$.

\end{lemma}

\begin{proof}
We define $\tilde{v}$ on successively larger domains, by beginning with $Y$ (where it must coincide with $v$), and adding on sets $F_n, \dots, F_0$.
Extending the definition of $\tilde{v}$ to add the set $F_i$ can be done by applying the special case of the lemma where $F_n = X$; so, let us simply prove this case.

We are given $v$ such that $v^*v = q|_Y$ and $vv^* \geq p_n|_Y$.
That is, $v^*p_nv = p_n|_Y$ and $p_nvv^*p_n \leq q|_Y$.
So, by applying Lemma \ref{Phillips-embedding},  we may extend $p_nv$ to a continuous partial isometry $w_1$ on $X$ such that $w_1^*w_1 = p_n$ and $w_1w_1^* \leq q$.

We then have $v^*(1-p_n)v = (q - w_1w_1^*)|_Y$ and $(1-p_n)vv^*(1-p_n) \perp p_n|_Y$.
So by applying Theorem \ref{DimFinite-Embedding} again, we may extend $(1-p_n)v$ to a continuous partial isometry $w_2$ on $X$ such that $w_2^*w_2 = q - w_1w_1^*$ and $w_2w_2^* \perp p_n$.
Finally, set $\tilde{v} = w_1 + w_2$.
\end{proof}

\begin{prop} \alabel{Dim3-DataAttained}
Suppose we are given a lower semicontinuous function $r\in \mathrm{Lsc}_\sigma(X,\mathbb{N}\cup\{\infty\})$, and for
each $i < \infty$ a Hilbert module $M_i$ on $R_{=i} := r^{-1}(\{i\})$, of constant dimension $i$.
Assume that $r$ has gaps of at least $(\dim X - 1)/2$ (this automatically holds if $\dim X \leq 3$). Then there exists a countably generated Hilbert module $M$ on $X$ such that
$\dim M=r$ and $M|_{R_{=i}} \iso M_i$ for each $i$.
\end{prop}

\begin{proof}
Let us only work with $n_i$ such that $R_{n_i} \neq \emptyset$, so that
\[ n_{i+1} \geq n_i + \frac{\dim X - 1}{2}. \]
For every such $n_i$, there is an open set $U_{n_i}$  such that $R_{=n_i} \subseteq U_{n_i}$ and $M_{n_i}$ extends to $U_{n_i}$. Let us first prove the theorem assuming that the sets $(U_{n_i})$ satisfy that
\begin{align}\label{limsupcond}
\limsup U_{n_i} = R_{=\infty}.
\end{align}
We will then indicate how the sets $(U_{n_i})$ may be chosen so that the preceding
condition holds.

Given the sets $(U_{n_i})$ as indicated above,
let us use Lemma \ref{ShrinkingLemma} to obtain sets $A_{n_i}$ which are relatively closed in $R_{\geq n_i}$, such that
\[ R_{\geq n_i} = \bigcup_{j \geq i} \interior{A_{n_j}}, \]
and $M_{n_i}$ extends to $A_{n_i}$. Let us set $A_0 = X$.

We let $q_{n_i}$ be a continuous, rank $n_i$ projection defined on $A_{n_i}$ such that, by restricting to $R_{={n_i}}$, it gives a Hilbert module isomorphic to $M_{n_i}$.
Let us produce a rank-ordered family $(p_{n_i}, A_{n_i})_{i=0}^\infty$, such that $p_{n_i}$ is Murray-von Neumann equivalent to $q_{n_i}$ for each $i$. This will prove the proposition.

We will obtain the $p_{n_i}$'s inductively, with $i$ beginning at $0$ and increasing.
Let us set $p_0=0$. Given $p_0, \dots, p_{n_{i-1}}$, let us now construct $p_{n_i}$.
The sets $A_{n_i} \cap A_{n_j}$ are relatively closed in $A_{n_i}$ for $j < i$.
Thus, by Lemma \ref{DimFinite-Engulfing}, we may find $p_{n_i}$ which is Murray-von Neumann equivalent to $q_{n_i}$, and satisfies for each $x \in A_{n_i}$
\[ p_{n_i}(x) \geq \bigvee_{j \mid x \in A_{n_i} \cap A_{n_j}} p_{n_j}(x). \]
This is exactly the compatibility requirement \eqref{ROF-Compatibility} for a rank-ordered family of projections.

It remains to show that the modules $M_{n_i}$ may be extended to open sets $U_{n_i}$
satisfying \eqref{limsupcond}. For every $i$, let $W_{n_i}$ be an open set such that $R_{=n_i}\subseteq W_{n_i}\subseteq R_{\geq n_i}$ and $M_i$ extends to $W_{n_i}$.
Let $\lambda_{n_i}\colon R_{\geq n_i}\to [0,1]$ be a continuous function such that
$\lambda_{n_i}|_{R_{=n_i}}=1$ and $\lambda_{n_i}|_{(W_{n_i})^c}=0$.

Consider the vector of functions
\[
(\lambda_{n_1}|_{R_{\geq n_{d+1}}}, \lambda_{n_{2}}|_{R_{\geq n_{d+1}}},\dots,\lambda_{n_{d+1}}),
\]
where $d=\dim X<\infty$. This vector defines a continuous map from $R_{\geq n_{d+1}}$ to $[0,1]^{d+1}$. Let $\epsilon\in C_0(R_{\geq n_{d+1}})^+$ be such that $0<\epsilon(x)<1/2$ for all
$x\in R_{\geq n_{d+1}}$. Since $\dim R_{\geq n_{d+1}}\leq d$, by \cite[Lemma 3.1]{benoit} there are perturbations $\tilde \lambda_{n_{k}}$ of the functions $\lambda_{n_{k}}|_{R_{\geq n_{d+1}}}$, for $k=1,\dots,d+1$, such that
\[
\left|\tilde \lambda_{n_{k}}-\lambda_{n_{k}}|_{R_{\geq n_{d+1}}}\right|<\epsilon,\hbox{ for }k=1,\dots,d+1,
\]
and such that there is no $x\in R_{\geq n_{d+1}}$ for which $\tilde \lambda_{n_{k}}(x)=1/2$ for all $k=1,\dots,d+1$. Notice that since the function $\epsilon$ vanishes outside
$R_{\geq n_{d+1}}$, the functions $\tilde \lambda_{n_k}$ extend continuously
to $R_{\geq n_{k}}$, and agree with $\lambda_{n_{k}}$ on $R_{\geq n_{k}}\backslash R_{\geq n_{d+1}}$, for $k=1,\dots,d+1$. We have
\[
R_{=n_k}\subseteq \tilde \lambda_{n_k}^{-1}((\frac{1}{2},1])\subseteq W_{n_{k}}\]
for $k=1,\dots,d+1$. Thus, the module $M_{n_k}$ extends to $\tilde \lambda_{n_k}^{-1}((1/2,1])$. For $k=1,\dots,d+1$, let us set
\[
U_{n_k}=\tilde \lambda_{n_k}^{-1}([0,\frac{1}{2}))\cup \tilde\lambda_{n_k}^{-1}((\frac{1}{2},1]),\]
and extend $M_{n_k}$ to $U_{n_k}$ by setting it equal to an arbitrary module of constant dimension $n_k$ on the set   $\tilde \lambda_{n_k}^{-1}([0,1/2))$. The open sets $U_{n_k}$ obtained in this way satisfy that
\[
R_{\geq n_{d+1}}\subseteq U_{n_1}\cup U_{n_2}\cup\dots U_{n_{d+1}}.
\]
We continue finding the sets $U_{n_k}$, for $k=d+2,\dots,2(d+1)+1$, in the same way, and so on.
The resulting sequence of open sets satisfies \eqref{limsupcond}.
\end{proof}

Theorem \ref{Dim3-Embedding-Iso} (ii) and Proposition \ref{Dim3-DataAttained} together form a computation of the isomorphism classes of countably generated Hilbert modules with a prescribed dimension function, when $\dim X \leq 3$ or the dimension function has large gaps.
In \cite[Proposition 10]{Dupre:ClassII}, Dupr\'{e} found this computation for the large gaps situation, under the condition that the dimension function is bounded.
Our result improves Dupr\'{e}'s most notably in that we also describe the conditions for embedding (in Theorem \ref{Dim3-Embedding-Iso} (i)).
\section{Cuntz comparison of Hilbert modules\alabel{CEI}}
In \cite{CowardElliottIvanescu} Coward, Elliott and Ivanescu introduced 
a preorder relation among Hilbert C$^*$-modules in order to
describe the Cuntz semigroup of a C$^*$-algebra using Hilbert C$^*$-modules.
Let us recall this relation here.

For a submodule $F$ of a Hilbert module $H$, let us write $F\waysubset H$
if there is $T\in K(H)^+$ such that $Tx=x$ for all $x\in F$. The Cuntz comparison
of Hilbert modules is defined as follows.
\begin{defn}
Let $M$ and $N$ be Hilbert modules over  a C$^*$-algebra $A$. Then $M\CEIleq N$
if for every $F\waysubset M$ there is $F'\waysubset N$ such that $F\cong F'$.
We write $M\CEIeq N$ if $M\CEIleq N$ and $N\CEIleq M$. 
\end{defn}

Embedding and isomorphism are stronger relations than Cuntz comparison
and equivalence (see Example \ref{telescope} below). This weakening allows for more flexibility 
in the resulting comparison theory.

Let us now consider modules over a commutative C$^*$-algebra. By the correspondence
between countably generated Hilbert modules and pointwise range projections, we can apply 
the relations $\CEIleq$ and $\CEIeq$ to pointwise range projections. 
For a pointwise range projection $p$ lying below another pointwise range projection $q$ we write $p\waysubset q$ if $M_p\waysubset M_q$.
Directly expressed in terms of the projections $p$ and $q$, we have $p\waysubset q$ if there is $a\in C_0(X,\mathcal K)^+$
such that $ap=p$ and $qa=a$. We have that $p\CEIleq q$ if for every $p'\waysubset p$ there is $q'$
such that $p'\cong q'\waysubset q$.

\begin{lemma}\alabel{CCCC}
If $p$ and $q$ are pointwise range projections such that $p\waysubset q$ then
$R_{\geq i}(p) \waysubset R_{\geq i}(q)$ for each $i \geq 1$.
\end{lemma}

\begin{proof}
Let $a\in C_0(X,\mathcal K)^+$ be such that $ap=p$ and $aq=a$.
We have $0\neq p(x)\leq a(x)$ for $x\in R_{\geq 1}(p)$. Thus, $\|a(x)\|\geq 1$ on $R_{\geq 1}(p)$ and so
$\closure{R_{\geq 1}(p)}$ is compact.

Let us show that $\overline{R_{\geq i}(p)} \subseteq R_{\geq i}(q)$ for all $i\geq 1$ (since $\overline{R_{\geq 1}(p)}$ is compact, this suffices to complete the proof). Let $x\in \overline{R_{\geq i}(p)}$. Choose $y\in R_{\geq i}(p)$ such that
$\|a(x)-a(y)\|<1$. Then  $\|a(x)p(y)-p(y)\|<1$. Hence
\[
i=\rank p(y)\leq \rank (a(x)p(y))\leq \rank a(x)\leq \rank q(x).
\]
Thus, $x\in R_{\geq i}(q)$. 
\end{proof}

\begin{lemma}\alabel{CEIprojections}
If $p$ and $q$ are continuous projections then $p\CEIleq q$ if and only
if for any compact subset $K$ of $X$ we have $p|_{K}\preceq q|_{K}$.
\end{lemma}

\begin{proof}
If $p\CEIleq q$ then this relation is passed on to the restrictions of $p$ and $q$ to any closed subset of
$X$. We thus have $p|_{K}\CEIleq q|_{K}$ for any compact $K$. Since $p|_{K}\waysubset p|_{K}$ (choose $a=p|_K$)
we get that $p|_{K}\preceq q|_{K}$.

Suppose on the other hand that $p|_{K}\preceq  q|_{K}$ for any compact $K$. Let $p'$ be
a pointwise range projection with $p'\waysubset p$. Then $R_{\geq 1}(p')\waysubset X$, and so 
\[
p'|_{R_{\geq 1}(p')}\leq p|_{R_{\geq 1}(p')}\preceq q|_{R_{\geq 1}(p')}. 
\]
Thus, $p'\preceq q'\waysubset q$, where $q'$ is the pointwise range projection equal to $q$
on $R_{\geq 1}(p')$ and 0 on the complement of this set.
\end{proof}

The results of the Section \ref{Dim3} have the following consequences for the Cuntz semigroup.
\begin{cor}\alabel{Dim3-CEI}
 Let $M$ and $N$ be countably generated Hilbert modules over $C_0(X)$. Suppose that either 
$\dim M$ has gaps of at least $\dim X/2$ or $\dim X\leq 3$. Then $M\CEIleq N$ if and only
if $\dim M\leq \dim N $ and 
\begin{align}\alabel{CEIrestrictions}
M|_{R_{=i}(M)\cap R_{=j}(N)} \CEIleq N|_{R_{=i}(M)\cap R_{=j}(N)}
\end{align}
for all $i,j=0,1,2\dots$.
\end{cor}

\begin{remark}
In view of Lemma \ref{CEIprojections}, the condition  \eqref{CEIrestrictions} may be restated
as $M|_K\CEIleq N|_K$ for any compact $K\subseteq R_{=i}(M)\cap R_{=j}(N)$.
\end{remark}

\begin{proof}
It is clear that if $M \CEIleq N$ then $M|_Y \CEIleq N|_Y$ for any $Y$ that is the intersection of a closed and open set.

For the converse, we may assume that $M = M_p$ and $N = M_q$ for some pointwise range projections $p$ and $q$.
If $p' \waysubset p$ then by Lemma \ref{CCCC}, we must have $R_{\geq i}(p') \waysubset R_{\geq i}(p)$ for each $i$, and so $R_{=i}(p') \cap R_{=j}(q)$ is pre-compact in $R_{=i}(p) \cap R_{=j}(q)$.
From this and Lemma \ref{CEIprojections}, we can verify that $p'$ and $q$ satisfy the hypotheses of Theorem \ref{Dim3-Embedding-Iso} (i), whence $p'\preceq q$.
Since $p' \waysubset p$ was arbitrary, this shows that $p \CEIleq q$.

\end{proof}

The following example shows that Cuntz equivalence and isomorphism differ even
for continuous projections.
\begin{example} \alabel{telescope} 
Let $X$ be the disjoint union $\bigsqcup_{i=1}^\infty \mathbb{T}\times [i,i+1]$ module the identification of
the point  $(z,i+1)\in  \mathbb T\times [i,i+1]$ 
with $(z^2,i+1)\in  \mathbb T\times [i+1,i+2]$, for $i \in \mathbb{N}$ and $z\in \mathbb{T}$.
Let $K_n$ be the image in this quotient of the set 
$\bigsqcup_{i=1}^{n-1} \mathbb{T}\times [i,i+1]\sqcup  (\mathbb{T}\times [n,n+1])$.
It is shown in \cite[Example 3F9]{Hatcher} that  while $H^2(K_n)=0$ for all $n$, $H^2(X)$ is uncountable. 
By the correspondence between line bundles and elements of $H^2(X)$ (via the first Chern class, see \cite[Theorem 3.4.16]{husemoller}), there are uncountably many non-isomorphic line bundles on $X$. These give rise to uncountably many Murray-von Neumann classes of continuous rank 1 projections on $X$. Let us show that they are all Cuntz equivalent. Let $p$ and $q$ be rank 1 continuous projections on $X$. Since $H^2(U_n)=0$, we have $p\cong q$ on $U_n$ for all $n$. Thus, $p\CEIeq q$. 
Notice that if $p$ and $q$ are continuous rank 1 projections and $p\ncong q$ then we do not have $p\preceq q$ nor 
$q\preceq p$. Thus, in this case, the modules $M_p$ and $M_q$ do not embed in each other. This answers a question raised in \cite[Page 162]{CowardElliottIvanescu}.

In \cite[Section 4]{BrownCiuperca}, two Hilbert modules (over a stably finite C$^*$-algebra) are found which are Cuntz equivalent but not isomorphic, also showing how Cuntz equivalence differs from isomorphism.
However, unlike the present example the modules in \cite{BrownCiuperca} do embed into each other.

\end{example}

\subsection{A Description of \texorpdfstring{$\Cu(C_0(X))$}{Cu(Co(X))} for \texorpdfstring{$\dim X\leq 3$}{dim X at most 3}}\alabel{Dim3-Description}
Let us review the description of the Cuntz semigroup, $\Cu(A)$, in terms of Hilbert $A$-modules, as given in \cite{CowardElliottIvanescu}.
Taking the equivalence classes of countably generated Hilbert $A$-modules under the relation $\CEIeq$ gives a set upon which $\CEIleq$ induces an order.
An addition operation may be defined by
\[ [M] + [N] := [M \dsum N]. \]
The resulting ordered semigroup is called the Cuntz semigroup, and is denoted by $\Cu(A)$.

Here, we will obtain a description of the Cuntz semigroup of $C_0(X)$ where $X$ has dimension at most three.
We will define an ordered semigroup $\specialCuntz(X)$ and show that it can be identified with $\Cu(C_0(X))$.

Let us define an equivalence relation on continuous projections on $X$ given by $p$ is equivalent to $q$ if $p|_K \mvneq q|_K$ for all $K \subseteq X$ compact.
Let $\specialV^n(X)$ denote the set of equivalence classes of projections which have constant rank $n$.
Then by Lemma \ref{CEIprojections}, when $X$ is $\sigma$-compact or $n=0$, $\specialV^n(X)$ can be identified with the elements of $\Cu(C_0(X))$ which have constant rank $n$.

Let $\specialCuntz(X)$ consist of pairs $(r, (\rho_i)_{i=0}^\infty)$, where $r \in \mathrm{Lsc}_\sigma(X)$ and $\rho_i \in \specialV^i(r^{-1}\{i\})$ for each $i$.

To make $\specialCuntz(X)$ a semigroup, we shall define an order relation and an addition operation as follows.
Let $(r, (\rho_i)_{i=0}^\infty), (r', (\rho_i')_{i=0}^\infty) \in \specialCuntz(X)$.

\textit{Ordering.} $(r, (\rho_i)_{i=0}^\infty) \leq (r', (\rho_i')_{i=0}^\infty)$ if $r \leq r'$ and for each $i$,
\[ \rho_i|_{r^{-1}\{i\} \cap r^{\prime -1}\{i\}} = \rho_i'|_{r^{-1}\{i\} \cap r^{\prime -1}\{i\}}. \]

\textit{Addition.} $(r, (\rho_i)_{i=0}^\infty) + (r', (\rho_i')_{i=0}^\infty) := (r + r', (\sigma_i)_{i=0}^\infty)$; $\sigma_i$ will be defined shortly.
Note that $(r + r')^{-1}\{i\}$ decomposes into components as
\[ (r + r')^{-1}\{i\} = \left(r^{-1}\{0\} \cap r^{\prime -1}\{i\}\right) \dunion \cdots \dunion \left(r^{-1}\{i\} \cap r^{\prime -1}\{0\}\right), \]
so that $\sigma_i$ is determined by its restriction to each set $r^{-1}\{j\} \cap r^{\prime -1}\{i-j\}$.
On $r^{-1}\{j\} \cap r^{\prime -1}\{i-j\}$, $\sigma_i = \rho_j + \rho_{i-j}'$.

\begin{prop}
Let $X$ be a locally compact Hausdorff space of dimension at most three.
Then $\Cu(C_0(X))$ is isomorphic, as an ordered semigroup, to $\specialCuntz(X)$, via the map $\Phi\colon\Cu(C_0(X)) \to \specialCuntz(X)$ given by
\[ \Phi(\alpha) = (\rank \alpha, (\alpha|_{R_{=i}(\alpha)})_{i=0}^\infty). \]
\end{prop}

\begin{proof}
For $\alpha, \beta \in \Cu(C_0(X))$, if $\alpha \leq \beta$ then $\rank \alpha \leq \rank \beta$ and 
\[ \alpha|_{R_{=i}(\alpha) \cap R_{=i}(\beta)} \leq \beta|_{R_{=i}(\alpha) \cap R_{=i}(\beta)}. \]
That is, representing $\alpha|_{R_{=i}(\alpha) \cap R_{=i}(\beta)}$ by the constant rank projection $p$ and $\beta|_{R_{=i}(\alpha) \cap R_{=i}(\beta)}$ by $p'$, we have by Lemma \ref{CEIprojections} that $p|_K \mvnleq p'|_K$ for each $K \subseteq X$ compact.
But since $p,p'$ both have constant rank $i$, this implies that $p|_K \mvneq p'|_K$ for each such $K$, and thus
\[ \alpha|_{R_{=i}(\alpha)} \cap R_{=i}(\beta) = \beta|_{R_{=i}(\alpha)} \cap R_{=i}(\beta). \]
Hence, $\Phi(\alpha) \leq \Phi(\beta)$.

Conversely, if $\Phi(\alpha) \leq \Phi(\beta)$ then by Corollary \ref{Dim3-CEI}, we have $\alpha \leq \beta$.

To see $\Phi$ is onto, let $(r, (\rho_i)_{i=0}^\infty) \in \specialCuntz(X)$.
Then for each $i$, there exists a Hilbert $C_0(r^{-1}\{i\})$-module $M_i$ such that $[M_i] = \rho_i$.
By Proposition \ref{Dim3-DataAttained}, there exists a Hilbert module $M$ such that $\dim M = r$ and $M|_{r^{-1}\{i\}} \iso M_i$
In particular, if $\alpha$ is the Cuntz element defined by this rank-ordered family, then $\Phi(\alpha) = (r, (\rho_i)_{i=0}^\infty)$.
Hence, $\Phi$ is an order isomorphism.

Finally, it is clear by the definition of $\Phi$ that it preserves addition.
\end{proof}

We have obtained a description of $\Cu(C_0(X))$ in terms of the sets $\specialV^n(Y)$ for the $\sigma$-compact subsets of $X$ which arise as the intersection of a closed set with an open set.
Note that, in turn, $\specialV^n(Y)$ can be described with \v{C}ech cohomology of compact subsets.
For a continuous projections $p$ of constant rank $n$, by Lemma \ref{Phillips-embedding} we may decompose
\[ p = \theta_p \dsum p', \]
where $\theta_p$ is trivial with rank $n-1$ and $p'$ has rank $1$.
For $K \subseteq Y$ compact, the isomorphism class of $p'|_K$ is determined by the Chern class $c_1(p'|_K) = c_1(p|_K) \in \check{H}^2(K)$ (using \v{C}ech cohomology) \cite[Theorem 16.3.4]{husemoller}.
Thus, we see that for $[p],[q] \in \specialV^n(Y)$, $[p] = [q]$ if and only if $c_1(p|_K) = c_1(q|_K)$ for all $K \subseteq Y$ compact.
Letting
\[ \limcoh{Y} = \varprojlim_{K \text{compact}, K \nearrow Y} \check{H}^2(K), \]
we apparently have an injective map $\limchern\colon\specialV^n(Y) \to \limcoh{Y}$ given by $\limchern([p]) = (c_1(p|_K))_K$.

Moreover, $\limchern$ is surjective, as we now show.
Let $(\gamma_K)_K \in \limcoh{Y}$.
Since $Y$ is $\sigma$-compact and locally compact, let $(K_i)_{i=1}^\infty$ be an increasing sequence of compact subsets such that $K_i \subseteq \interior{K_{i+1}}$ and $Y = \bigcup_{i=1}^\infty K_i$.
Since we can find a $\sigma$-compact open set $V$ such that $K_{i-1} \subseteq V \subseteq K_i$,  we may assume without loss of generality that $\interior{K_i}$ is $\sigma$-compact for all $i$.
Since the Chern class $c_1$ is surjective (by \cite[Theorem 16.3.4]{husemoller}), for each $i$, let $p_i$ be a continuous projection defined on $K_i$ such that $\gamma_{K_i} = c_1(p_i)$.
Let $q_i \in \RP(Y)$ be given by $q_i|_{\interior{K_i}} = p_i|_{\interior{K_i}}$ and $q_i|_{Y \setminus \interior{K_i}} = 0$.
We can easily see that $q_i \CEIleq q_{i+1}$, and so by \cite[Theorem 1 (i)]{CowardElliottIvanescu}, we may define
\[ \alpha = \sup [q_i] \in \Cu(C_0(Y)). \]
Then for each $i$, by taking the tail $([q_j])_{j > i}$, we see that $\alpha|_{K_i} = \sup [p_i] = [p_i]$.
Thus, $\alpha$ has constant rank $i$ and, since $(K_i)$ is cofinal, $\limchern(\alpha) = (\gamma_K)_K$.
\section{Further remarks\alabel{FurtherRemarks}}
\subsection{The clutching construction\alabel{Clutching}}
By Theorem \ref{Dim3-Embedding-Iso} and Corollary \ref{Dim3-CEI} the isomorphism and Cuntz equivalence classes of
a Hilbert module over a space of dimension at most 3 are determined by the restrictions of the module to
the subsets where its dimension is constant. Here we give an example of Hilbert modules over $S^4$
for which this fails. The example is based on the clutching construction given by
Dupr\'{e} in \cite[Page 319]{Dupre:ClassII}.

Let $X$ be a compact Hausdorff space. Let $SX$ denote its suspension.
We view $SX$ as the quotient space of $X\times [-1,1]$ obtained identifying
all the points in $X\times \{-1\}$ and the points in $X\times \{1\}$ (see \cite{husemoller}).
When speaking of subsets of $SX$, we use  the notation $X\times_{\sim} U$, with $U\subseteq [-1,1]$,
to refer to the image of $X\times U$ in the quotient.

Consider a Hilbert $C(SX)$-module $M$ with dimension $n$ on $X\times_\sim [-1,0]$ and dimension $m$ on $X\times_\sim (0,1]$.
Let $M \iso M_p$ for some range projection $p$, and let $(p_1,A_1),(p_2,A_2)$ be a rank-ordered family of projections for $p$, so that $p_1,p_2$ have ranks $m$ and $n$ respectively.
Necessarily, $A_2 = X\times_\sim (0,1]$, and by a possible shrinking, we may assume $A_1 = X\times_\sim [-1, \epsilon)$ for some $\epsilon > 0$.

Since the sets $X\times_{\sim} [-1,\epsilon)$ and $X\times_{\sim} (0,1]$ are contractible, $p_1$ and $p_2$ are trivial on these sets.
That is,  there are partial isometries
$v_1$ and $v_2$ such that $p_1=v_1^*v_1$, $v_1v_1^*=1_n$, $p_2=v_2^*v_2$ and $v_2v_2^*=1_m$.
Consider the continuous partial isometry $c_{v_1,v_2,\epsilon_0}\in \M_{m}(C(X))$ given by
\[
c_{v_1,v_2,\epsilon_0}(x)=v_2(x,\epsilon_0)v_1^*(x,\epsilon_0),
\]
where $\epsilon_0\in (0,\epsilon)$. Notice that $c_{v_1,v_2,\epsilon_0}^*c_{v_1,v_2,\epsilon_0}=1_n$.
Let us denote by $U_{n,m}(C(X))$ the set of partial isometries $c\in \M_m(C(X))$ such that $c^*c=1_n$. 

\begin{prop} (\cite[Section 4, Corollary 2]{Dupre:ClassII}) \alabel{Clutching-Iso-Homot}
The map $[M_p]\mapsto [c_{v_1,v_2,\epsilon_0}]$ is a well-defined bijection from the isomorphism classes of Hilbert $C(SX)$-modules
with dimension  $n$ on $X\times_\sim [-1,0]$, and dimension $m$ on $X\times_\sim (0,1]$, to the path connected components of  $U_{n,m}(C(X))$.
\end{prop}

\begin{example}\label{clutching-sphere}
Say $X=S^3$. Then $SX=S^4$. Let $S^4_+$ denote an open hemisphere of $S^4$
and $S^4_-$ its complement. By the previous proposition, the isomorphism classes
of Hilbert modules on $S^4$ that have constant rank 1 on $S^4_-$
and constant rank 2 on $S^4_+$ are in bijection with the homotopy classes
of partial isometries  $c\in I_{1,2}(C(S^3))$. For every $x\in S^3$, the elements
$c(x)\in \M_2(\mathbb{C})$ such that $c^*(x)c(x)=1_1$ correspond to the points in the unit sphere of
$\mathbb{C}^2$, i.e., $S^3$. Thus, the partial isometry $c$ may be viewed as a map from $S^3$ to $S^3$.
Such a map is classified, up to homotopy, by its degree. Thus, there is one isomorphism class
for every integer. Notice, on the other hand, that the modules corresponding to these isomorphism classes all satisfy
that their restrictions to $S^4_-$ and $S^4_+$---i.e., the sets where their dimension is constant---are pairwise isomorphic (since the hemispheres of the sphere are contractible).
\end{example}

In the next proposition we show that, for the Hilbert modules covered by Proposition
\ref{Clutching-Iso-Homot}, Cuntz equivalence agrees with isomorphism (and so, Example \ref{clutching-sphere} shows that the Cuntz class of a Hilbert $C_0(X)$-module may not by determined by its restrictions to the sets of constant dimension if $\dim X\geq 4$).

\begin{prop}
The homotopy class of $c_{v_1,v_2,\epsilon_0}$ depends only on the Cuntz class of 
$M_{p_1\vee p_2}$.
\end{prop}

\begin{proof}
Suppose that $p_1\vee p_2\CEIeq q$ for some projection $q \in \RP(SX)$, and $q$ has rank $n$ on $X\times_\sim [-1,0]$ and rank $m$ on $X\times_\sim (0,1]$.
Then for $\epsilon'\in (0,\epsilon)$ there is a rank-ordered
family of partial isometries
\[(z_1,X\times_{\sim} [-1,\epsilon)),(z_2,X\times_{\sim} (\epsilon',1])\]
such that $z_1^*z_1=p_1$ and $z_1z_1^*\leq q$ on  $X\times_{\sim} [-1,\epsilon)$,
and $z_2^*z_2=p_2$ and $z_2z_2^*=q$ on $X\times_{\sim} (\epsilon',1])$.
Set $z_1z_1^*=q_1$ and $q|_{X\times_{\sim} (0,1])}=q_2$. Then
\[
(q_1,X\times_{\sim} [-1,\epsilon)),(q_2,X\times_{\sim} (0,1]) 
\]
is a rank-ordered family for $q$.
Let $w_1$ and $w_2$ trivializations for $q_1$ and $q_2$. Choose $\epsilon_0\in (\epsilon',\epsilon)$.
Set $z_1w_1=v_1'$ and $z_2w_2=v_2'$. We have
\[
c_{w_1,w_2,\epsilon_0}(x)=(w_2^*w_1)(x,\epsilon_0)=
(v_2'(v_1')^*)(x,\epsilon_0)=c_{v_1',v_2',\epsilon_0}(x).
\]
The partial isometries $v_1'$ and $v_2'$
are trivializations for $p_1$ and $p_2$ on the sets $X\times_{\sim}[-1,\epsilon)$ and
$X\times_{\sim} (\epsilon',1]$ respectively.
Suppose that $v_1,v_2$ are any trivializations for $p_1$ and $p_2$.
The unitaries  $(v_1'v_1^*)(\cdot,\epsilon_0)\in M_n(C(X))$ and $(v_2(v_2')^*)(\cdot,\epsilon_0)\in M_m(C(X))$ are connected to constant unitaries (on $X$) by the paths $t\mapsto (v_1'v_1^*)(\cdot,t)$, $t\in [-1,\epsilon_0]$, and $t\mapsto (v_2(v_2')^*)(\cdot,t)$, $t\in [\epsilon_0,1]$. These constant unitaries are in turn 
connected to $1_n$ and $1_m$ respectively. Thus, $c_{v_1',v_2',\epsilon_0}$ is homotopic to $c_{v_1,v_2,\epsilon_0}$, as required.
\end{proof}

\subsection{The group \texorpdfstring{$K_0^*(C_0(X))$}{Ko*(Co(X))}\alabel{K0*}}
It is in \cite{Cuntz:dimensionfunctions} that Cuntz laid the groundwork for what would later be called the Cuntz semigroup.
The invariant that interested Cuntz there (and which he denoted by $K_0^*(A)$) is the enveloping group of the unstabilized Cuntz semigroup, generated by the unstabilized Cuntz semigroup in the same way that $K_0(A)$ is generated by the Murray-von Neumann semigroup of $A$.
(The unstabilized Cuntz semigroup is the subsemigroup of $\Cu(A)$ containing only those elements that can be represented by finitely generated submodules of $A^n$ for some $n$; it is denoted $W(A)$.  The terminology ``unstabilized'' is justified by the fact that $\Cu(A) = W(\K \tens A)$.)

Here, we find a description of $K_0^*(C_0(X))$ for finite dimensional $X$.
It turns out that introducing cancellation destroys both types of non-triviality that we've seen: that arising from non-trivial constant rank projections, and the more subtle nontriviality in how the constant rank pieces fit together, as seen in Example \ref{clutching-sphere}.

Following \cite[Section 5]{Dupre:ClassI}, we call a Hilbert $C_0(X)$-module $M$ elementary if it is isomorphic to one of the form
\[ \bigdsum_{i=1}^n C_0(U_i), \]
for some ($\sigma$-compact) open sets $U_i$. If $U$ and $V$ are $\sigma$-compact open sets, then 
\[ C_0(U) \dsum C_0(V) \iso C_0(U \cup V) \dsum C_0(U \cap V), \]
by \cite[Corollary 2]{Robert:Dim2}. For a general elementary Hilbert $C_0(X)$-module $M$, repeated application of this result shows that
\[ M \iso \bigdsum_{i=1}^{n} C_0(R_{\geq i}(M)). \]
Thus, the isomorphism class of an elementary Hilbert $C_0(X)$-module $M$ depends only on the function $\dim M$.

\begin{lemma}\alabel{ElementarySum}
Let $M$ be a Hilbert $C_0(X)$-module such that $\dim M$ is bounded.
The following are equivalent.

(i) For all $i$, $M|_{R_{=i}(M)}$ has finite type (as a vector bundle).

(ii) There exists an elementary module $N$ such that $M \dsum N$ is also elementary.
\end{lemma}

\begin{proof}
(ii) $\Rightarrow$ (i):
If $M \dsum N$ is elementary then in particular, it can be embedded into $C_0(X)^{\dsum n}$ for some $n$.
It follows that $M|_{R_{=i}(M)}$ embeds into a trivial bundle for each $i$.
By \cite[Proposition 3.5.8]{husemoller}, this shows that $M|_{R_{=i}(M)}$ has finite type.

(i) $\Rightarrow$ (ii):
Since $M_{R_{=i}(M)}$ has finite type for all $i$, there are finitely many open sets covering $R_{=i}(M)$ such the restriction to each of these sets is a trivial vector bundle.
To begin, we shall construct $N_0$ such that each constant rank set of $N_0$ (and in fact, of $M \dsum N_0$) is contained in a set where $M$ is trivial.

To do this, let $V_1, \dots, V_{n_1}$ be open sets, contained in $R_{\geq 1}$, such that $M|_{R_{=1} \cap V_i}$ is trivial for each $i = 1, \dots, n_1$.
Likewise, let $V_{n_{k-1} + 1}, \dots, V_{n_k}$ be open sets contained in $R_{\geq k}$ such that $M|_{R_{=k} \cap V_i}$ is trivial for each $i = n_{k-1} + 1, \dots, n_k$. 
Letting $m$ to be the maximum dimension of the fibres of $M$, set
\[ N_0 = \bigdsum_{i=1}^{n_m} C_0(V_i). \]
The constant rank sets of $N_0 \dsum C$ are exactly the same as those of $C$, and each is contained in some set $V_i \cap R_{=k}(M)$, for some $i$ between $n_{k-1}+1$ and $n_k$.
Thus, on each constant rank set, $M \dsum N_0$ corresponds to a trivial vector bundle.
Since $N_0$ is elementary, we have shown that we can reduce to the situation that the restriction of each constant rank set of $M$ is trivial.

Assuming that the restriction of $M$ to each constant rank set is trivial, let us show by induction on the maximum fibre dimension of $M$ that there exists an elementary module $N$ such that $M \dsum N$ is also elementary.
Of course, if $M$ only has fibres of dimension $0$ then $M = 0$.

For the inductive step, suppose that $\dim M$ is bounded by $m$.
By induction, there exists an elementary $C_0(R_{\leq m-1}(M))$-module $N_0$ such that $M|_{R_{\leq m-1}(M)} \dsum N_0$ is elementary.
Let $M'$ be the elementary Hilbert module whose dimension function is the same as that of $M \dsum N_0$.
Since $M'|_{R_{\leq m}(M)} \iso (M \dsum N_0)|_{R_{\leq m}(M)}$, \cite[Proposition 1]{Robert:Dim2} shows that
\[ M' \dsum (M \dsum N_0)|_{R_{=m}(M)} \iso M \dsum N_0 \dsum M'|_{R_{=m}(M)}. \]
The left-hand side is elementary, as is the right-hand summand $N_0 \dsum M'|_{R_{=m}(M)}$, which we may take as $N$.
\end{proof}

\begin{remark}
When $X$ is finite dimensional, all Hilbert $C_0(X)$-modules with bounded fibre dimension satisfy condition (i) of Lemma \ref{ElementarySum} (this follows from \cite[Proposition 3.5.8]{husemoller}.
\end{remark}

\begin{thm}
Let $X$ be a finite dimensional locally compact Hausdorff space, and let $\hat X$ be its one-point compactification.
Then $K_0^*(C_0(X))$ may be identified with the group of bounded maps $f\colon\hat X \to \mathbb{Z}$ satisfying $f(\infty)=0$ and for which $f^{-1}(\{i\})$ is the difference of two $\sigma$-compact open sets, for all $i$.
If $X$ is $\sigma$-compact, then $K_0^*(C_0(X))$ may simply be identified with the group of bounded maps 
$f\colon X \to \mathbb{Z}$ for which $f^{-1}(\{i\})$ is the difference of two $\sigma$-compact open sets, for all $i$.
\end{thm}

\begin{proof}
For finitely generated Hilbert modules $M$ and $M'$, we have that $[M] = [M']$ in $K_0^*(C_0(X))$ if and only if $M \dsum N \CEIeq M' \dsum N$ for some finitely generated Hilbert module $N$.
Clearly, this can only happen if $\dim M = \dim M'$, and Lemma \ref{ElementarySum} shows that $\dim M = \dim M'$ is sufficient.
So, $K_0^*(C_0(X))$ can be identified with the group generated by functions $\dim M$ where $M$ is a finitely generated Hilbert $C_0(X)$-module.
Such functions are exactly the bounded, lower semicontinuous maps $f\colon X \to \mathbb{N}$ for which $f^{-1}(\{i, i+1, \dots\})$ is open and $\sigma$-compact for all $i \geq 1$.
In particular, such $f$ satisfies the condition that $f^{-1}(\{i\})$ is the difference of two $\sigma$-compact open sets for all $i \geq 1$.
Moreover, if we view $f$ as a function on $\hat X$ by setting $f(\infty) = 0$ then $f^{-1}(\{0\})$ is also the difference of two $\sigma$-compact open sets.
Thus, $f$ is a function as in the statement above.

Let us now check that the set of functions described forms a group---that is, that it is closed under addition.
Suppose that for $t=1,2$, we have $f_t\colon\hat X \to \mathbb{Z}$ such that $f_t^{-1}(\{i\})$ is the difference of two $\sigma$-compact open sets, and that both functions are bounded between $-K$ and $K$.
Then for each $i$,
\[ (f_1 + f_2)^{-1}(\{i\}) = \bigcup_{j=-K}^K f_1^{-1}(\{j\}) \cap f_2^{-1}(\{i-j\}). \]
The family of $\sigma$-compact open sets is closed under finite intersections and unions, and thus so is the family of sets which are the difference of two $\sigma$-compact open sets.
Hence, $(f_1 + f_2)^{-1}(\{i\})$ is the difference of two $\sigma$-compact open sets, so that $f_1 + f_2$ does lie in the set described.

Finally, let us show that every function described does occur in $K_0^*(C_0(X))$.
For this, it suffices to show that for every set $Y$ which is the difference of two $\sigma$-compact open sets, $\chii_Y$ occurs as $\dim M - \dim N$ for some countably generated Hilbert $C_0(X)$-modules $M,N$.
This is clear, since if $Y = U \setminus V$ where $U,V$ are $\sigma$-compact and open, then $M = C_0(U)$ and $N = C_0(U \cap V)$ will work.
\end{proof}

\subsection{An absorption theorem\alabel{Absorption}}

In this section, we shall prove the following.

\begin{thm}\alabel{OpenAbsorption}
Let $U$ be a $\sigma$-compact open subset of $X$ and
let $M$ be a countably generated Hilbert $C_0(X)$-module. Suppose that $M|_U\cong \ltwo(U)$. 
Then $M\cong M\oplus \ltwo(U)$.
\end{thm}

Before proving the theorem we need two simple lemmas.

\begin{lemma}\alabel{complemented}
Let $M$ be a countably generated Hilbert $C_0(X)$-module, and let $U,V$ be $\sigma$-compact open sets with $V$ compactly contained in $U$.
If $F$ be a submodule of $MC_0(V)$ that is a direct summand of $MC_0(U)$ then $F$ is a direct summand of $M$.  
\end{lemma}

\begin{proof}
Let us show that every $m\in M$ decomposes into the sum of one
element in $F$ and one in $F^{\perp}$.
Let $\lambda\in C_0(U)$ be such that $\lambda(x)=1$ on $V$.
Then $m=m(1-\lambda)+m\lambda$. The first
summand is orthogonal to $HC_0(V)$, whence belongs to $F^{\perp}$. The second summand 
belongs to $HC_0(U)$, and since $F$ is complemented in $HC_0(U)$, decomposes into the
sum of an element in $F$ and one in $F^{\perp}$.  
\end{proof}

\begin{lemma}\label{complementedsum}
Let $M$ be a countably generated Hilbert $C_0(X)$-module, let $(M_i)_{i=1}^\infty$ be a sequence of pairwise orthogonal, countably generated, such that $M_i+M_i^{\perp}=M$ for all $i$.
Suppose that for a sequence of generators  $(\xi_i)_{i=1}^\infty$ of $M$
we have that the series $\sum_{k=1}^\infty \xi_i^{k}$ is convergent for all $i$, 
where $\xi_i^{k}$ denotes the projection of $\xi_i$ onto $M_k$. Then the
submodule $\overline{\sum_{k=1}^\infty M_k}$ is a direct summand of $M$.
\end{lemma}
\begin{proof}
It is easily verified that the vector $\xi_i-\sum_{k=1}^\infty \xi_i^{k}$ is orthogonal 
to $M_k$ for all $k$. Thus, $\xi_i-\sum_{k=1}^\infty \xi_i^{k}$ is orthogonal to 
$\overline{\sum_{k=1}^\infty M_k}$. This shows that each of the vectors $\xi_i$ can be decomposed
in a sum of an element in $\overline{\sum_{k=1}^\infty M_k}$ and one orthogonal to $\overline{\sum_{k=1}^\infty M_k}$.
Taking linear combinations and passing to limits we get the same for all the vectors of $M$. 
\end{proof}

\begin{proof}[Proof of Theorem \ref{OpenAbsorption}]
It is enough to show that $\ltwo(U)$ is isomorphic to a direct summand
of $M$, for if $M\cong M'\oplus \ltwo(U)$ then adding $\ltwo(U)$ on both sides
we get $M\oplus \ltwo(U)\cong M'\oplus \ltwo(U)\oplus \ltwo(U)\cong M'\oplus \ltwo(U)\cong M$.

Let $(V_i)_{i=1}^\infty$ be an increasing sequence of open sets
compactly contained in $U$ and such that $U=\bigcup_i V_i$. Let $(\xi_i)_{i=1}^\infty$ be a sequence of
generators of $M$. We modify these generators as follows: define $\widetilde \xi_i=\xi_i(1-\lambda_i)$, where
$\lambda_i\in C_0(U)$ is equal to 1 on $V_i$. The new vectors $\widetilde\xi_i$ satisfy that 
$\widetilde\xi_j\perp MC_0(V_i)$ if $i\leq j$, and $M$ is spanned by $\{\widetilde \xi_1,\widetilde \xi_2,\dots\}\cup MC_0(U)$.  

Let us identify  $MC_0(U)$ with $\ltwo(U)$. We have  that $\ltwo(U)\cong \bigoplus_{i=1}^\infty \ltwo(V_i)$.
Notice that by Lemma \ref{complemented} each of the modules $\ltwo(V_i)$ is complemented in $M$.
 Choose an open set $V_i$. Consider the orthogonal projections of the vectors $(\widetilde\xi_j)_{j=1}^\infty$ onto
$\ltwo(V_i)$. Only a finite number of them are non-zero. By further decomposing $\ltwo(V_i)$ into a countable sum
of submodules, all isomorphic to $\ltwo(V_i)$, we can choose one of those summands such that the projections
of all the vectors $(\widetilde\xi_j)_{j=1}^\infty$ onto that summand have norm at most $\frac{1}{2^i}$ (and only finitely many 
are non-zero). Denote this submodule by $M_i$. Performing this construction for every $i$ we obtain a sequence of sumbmodules $(M_i)_{i=1}^\infty$ of $MC_0(U)$, such that $M_i\cong \ltwo(V_i)$ and $M_i$ is complemented in $M$ for all $i$, and 
the series $\sum_i \widetilde \xi_j^i$, of projections  of the vectors $\xi_j$ onto the $M_i$'s, is convergent for all $j$.
This is also true for all the vectors in $MC_0(U)$, since by construction $\overline{\sum_{i=1}^\infty M_i}$ is 
complemented in $MC_0(U)$. It follows by Lemma \ref{complementedsum} that  $\overline{\sum_{i=1}^\infty M_i}$
is complemented in $M$. Since $M_i\cong \ltwo(V_i)$ for all $i$, we have $\overline{\sum_{i=1}^\infty M_i}\cong \ltwo(U)$. This
completes the proof.
\end{proof}

\begin{cor}
Let $M$ and $N$ be $C_0(X)$-modules such that $M|_{U}\cong N|_{U}\cong \ltwo(U)$ and 
let $\phi\colon M|_{X\backslash U}\to N|_{X\backslash U}$ be an isomorphism of Hilbert modules. 
Then there is $\psi\colon M\to N$, isomorphism of Hilbert modules, such that $\psi|_{X\backslash U}=\phi$.
\end{cor}

\begin{proof}
By Theorem \ref{OpenAbsorption}, we have $M = M' \dsum L$ where $L \iso \ltwo(U)$.
It follows that the isomorphism
\[ M = M' \dsum L \iso M' \dsum L \dsum \ltwo(U) = M \dsum \ltwo(U) \]
fixes $M|_{X \setminus U}$.

By \cite[Theorem 2]{CiupercaRobertSantiago}, we have an isomorphism
\[ \psi'\colon M \dsum \ltwo(U) \to N \dsum \ltwo(U), \]
such that $\psi'|_{X \setminus U} = \phi$.
Combining this with isomorphisms which fix $M|_{X \setminus U}$ and $N|_{X \setminus U}$ gives
\[ \psi\colon M \to M \dsum \ltwo(U) \to N \dsum \ltwo(U) \to N \]
such that $\psi|_{X \setminus U} = \phi$.
\end{proof}
 
\begin{remark}
By \cite[Th\'eor\`eme 5]{DixmierDouady}, whenever $X$ has finite dimension, the condition $M|_U \cong \ltwo(U)$ is the same as $\dim M|_{U} = \infty$.
However, by \cite[Corollaire 1 after Th\'eor\`eme 6]{DixmierDouady}, this is not the case when $X$ has infinite dimension.
This last corollary confirms \cite[Conjecture 1]{Dupre:infinite} in the case that $A$ there is closed.
It also generalizes \cite[Proposition 12]{Dupre:ClassII} in two ways: first, it drops the restriction that $\dim M$ has finite range; second, it is the best possible generalization to the situation that $X$ is not finite dimensional (there, we must require that $M|_U \iso \ltwo(U)$ and not simply that $\dim M|_U = \infty$).
\end{remark}

\end{document}